\documentclass{article}
\usepackage{amsmath,amssymb,amsthm,graphicx,epsfig,float,url}
\usepackage[colorlinks=true,citecolor={Plum},linkcolor={Periwinkle}]{hyperref}
\usepackage{pdfsync}
\usepackage[usenames, dvipsnames]{xcolor}
\usepackage{tikz}
\usepackage{subfig}
\usepackage{bbm}
\usepackage{stmaryrd}
\usepackage{mathrsfs}  
\usepackage{caption}
\usepackage{float}
\usepackage{esint}
\usepackage[english]{babel}
\usepackage{dsfont}
\usepackage{graphicx}
\usepackage{grffile}

\usepackage[makeroom]{cancel}
\usetikzlibrary{patterns}
\usepackage{accents}

\newcommand{\R}{\textnormal{I\kern-0.21emR}}
\newcommand{\N}{\textnormal{I\kern-0.21emN}}

\renewcommand{\geq}{\geqslant}
\renewcommand{\leq}{\leqslant}

\def\e{{\varepsilon}}

\def\OT{{(0;T)\times \O}}

\allowdisplaybreaks

\def\YYint#1#2#3{{\setbox0=\hbox{$#1{#2#3}{\iint}$}
    \vcenter{\hbox{$#2#3$}}\kern-.51\wd0}}
 
\usepackage[dvipsnames]{xcolor}

\newtheorem*{theorem*}{Theorem}

\newtheorem{theorem}{Theorem}  
\newtheorem{proposition}{Proposition}

\newtheorem{lemma}{Lemma}

\theoremstyle{definition}\newtheorem{remark}{Remark}

\def\O{{\Omega}}
\def\n{{\nabla}}
\def\p{{\varphi}}

\usepackage{xargs}
 \usepackage[colorinlistoftodos,textsize=small]{todonotes}
 \newcommandx{\unsure}[2][1=]{\todo[linecolor=red,backgroundcolor=red!25,bordercolor=red,#1]{#2}}
 \newcommandx{\change}[2][1=]{\todo[linecolor=blue,backgroundcolor=blue!25,bordercolor=blue,#1]{#2}}
 \newcommandx{\info}[2][1=]{\todo[linecolor=green,backgroundcolor=green!25,bordercolor=green,#1]{#2}}
 \newcommandx{\improvement}[2][1=]{\todo[linecolor=yellow,backgroundcolor=yellow!25,bordercolor=yellow,#1]{#2}}
 
  \newcommandx{\biblio}[2][1=]{\todo[linecolor=blue,backgroundcolor=magenta!25,bordercolor=blue,#1]{#2}}

\begin{document}
\nocite{*}
\title{Optimisation of the total population size with respect to the initial condition for semilinear parabolic equations:\\ Two-scale expansions and symmetrisations}


\author{Idriss Mazari\footnote{Technische Universit\"{a}t Wien, Institute of Analysis and Scientific Computing, 8-10 Wiedner Haupstrasse, 1040 Wien (\texttt{idriss.mazari@tuwien.ac.at})}, Gr\'egoire Nadin\footnote{ CNRS, Sorbonne Universit\'es, UPMC Univ Paris 06, UMR 7598, Laboratoire Jacques-Louis Lions, F-75005, Paris, France (\texttt{gregoire.nadin@sorbonne-universite.fr})}, Ana Isis Toledo Marrero\footnote{ Sorbonne Universit\'es, UPMC Univ Paris 06, UMR 7598, Laboratoire Jacques-Louis Lions, F-75005, Paris, France (\texttt{ana-isis.toledo\_marrero@sorbonne-universite.fr)}}}

\maketitle

\begin{abstract}
In this article, we propose in-depth analysis and characterisation of the optimisers of the following optimisation problem: how to choose the initial condition $u_0$ in order to maximise the spatial integral at a given time of the solution of the semilinear equation $u_t-\Delta u=f(u)$, under $L^\infty$ and $L^1$ constraints on $u_0$? Our contribution in the present paper is to give a characterisation of the behaviour of the optimiser $\overline{u}_0$ when it does not saturate the $L^\infty$ constraints, which is a key step in implementing efficient numerical algorithms. We give such a characterisation under mild regularity assumptions by proving that in that case $\overline{u}_0$ can only take values in the "zone of concavity" of $f$. This is done using two-scale asymptotic expansions. We then show how well-known isoperimetric inequalities yield a full characterisation of maximisers when $f$ is convex. Finally, we provide several numerical simulations in one and two dimensions that illustrate and exemplify the fact that such characterisations significantly improve the computational time. All our theoretical results are in the one-dimensional case and we offer several comments about possible generalisations to other contexts, or obstructions that may prohibit doing so.
\end{abstract}


\medskip

\paragraph{Acknowledgement} The authors wish to warmly thank J. Pertinand for scientific conversations.   The authors also wish to thank the anonymous referee for her/his comments, which have helped us improve the quality of this paper. \color{black} I. Mazari was partially supported by the French ANR Project ANR-18-CE40-0013 - SHAPO on Shape optimisation and by the Austrian Science Fund (FWF) through the grant I4052-N32 . I. Mazari and G. Nadin were partially supported by the Project "Analysis and simulation of optimal shapes - application to lifesciences" of the Paris City Hall.


\section{Introduction}
\subsection{Scope of the article}
In this article, we propose to establish several results concerning an optimal control problem for a class of semilinear parabolic equations. Some aspects of this problem have been initially addressed by two of the authors in \cite{NadinToledo}. In the setting under consideration, the control variable to optimise is the initial condition. As we will see throughout the statement of the results, the form (e.g. convex or concave) of the semilinearity plays a crucial role in the analysis and calls for a detailed study of second order optimality conditions, which is our main result, Theorem \ref{Prop1}. In the case of convex semilinearities, using rearrangement arguments, we can give a full characterisation of maximisers, see Theorem \ref{Th:Convex}. Using Theorem \ref{Prop1}, we can improve an algorithm initially developed in \cite{NadinToledo}, and we display numerical results in Section \ref{sec:numerics}. 

\paragraph{Initial motivation of the paper}
The origin of this paper is the study of an optimal control problem that arises naturally in mathematical biology and that deals with bistable reaction-diffusion equations. Namely, for a semilinear equation, what is the best possible initial condition ("best" being understood as maximising the integral of the solution at a certain time horizon $T$)? The complicated behaviour of bistable non-linearities, which are neither convex nor concave, makes the analysis of this query very intricate. The two aforementioned results, Theorems \ref{Prop1} and \ref{Th:Convex}, enable us to show how complicated the behaviour of maximisers can be for such non-linearities. 
 Bistable equations are of central importance in mathematical biology \cite{Murray} and, very broadly speaking, model the evolution of a subgroup of a population. Among their many applications, one may mention chemical reactions \cite{Perthame2015}, neurosciences \cite{EvansNerve}, phase transition \cite{DeMasi1986},
linguistic dynamics \cite{Prochazka2017}
or the  evolution of diseases \cite{Murray}.
The last interpretation is of particular relevance to us, given that this model is used to design optimal strategies in order to control the spread of several mosquito borne diseases such as the dengue; this was the main motivation in \cite{MR3987527}. The strategy is to release a certain amount of Wolbachia carrying mosquitoes (Wolbachia is a bacterium that inhibits the transmission of mosquito borne diseases that individuals inherit from their mother) in a population of wild mosquitoes that can potentially transmit the diseases, in order to maximise the proportion of Wolbachia carrying mosquitoes at the final time. In mathematical terms: given a time horizon $T$,
\begin{center}
\emph{How should we arrange the initial population in order to maximise the population size at $T$?}\end{center}
Even without having stated it formally, we can make two observations on this problem: the first one is that , since the variable of the equation is the proportion of a subgroup, \color{black} we need to enforce pointwise ($L^\infty$) constraints. The second one is that we naturally have to add an $L^1$ constraint for modelling reasons. Both   of \color{black} these constraints can in practice be very complicated to handle.

\paragraph{Optimisation problems in mathematical biology}
Let us briefly sketch how this problem fits in the literature devoted to such optimisation and control problems for mathematical biology.
Optimisation problems for reaction-diffusion equations have by now gathered a lot of attention from the mathematical community. Most of these optimisation problems are set in a stationary setting, that is, assuming that the population has already reached an equilibrium, and the main problems that have been considered often deal with the optimisation of the spatial heterogeneity \cite{CaubetDeheuvelsPrivat,InoueKuto,KaoLouYanagida,Lou2006,LouNagaharaYanagida,MNP,MNP2,MRB,NagaharaYanagida}  (we also refer to the recent surveys \cite{Lam,MNP3}); most of these works deal with monostable nonlinearities. We also point to the recent \cite{BINTZ2020} for the study of an optimal control problem for parabolic monostable equations. On the other hand, optimisation problems for bistable equations, which are the other paradigmatic class of equations in mathematical biology \cite{Murray}, have received a less complete mathematical treatment, but are now the topic of an intense research activity from the control point of view, see \cite{MR3987527,NadinToledo} and the references therein. Related optimal control problems are not yet fully understood. More generally, less attention has been devoted to optimisation problem with respect to the initial condition for such semilinear evolution equations.

%

\subsection{Mathematical setup and statement of the results}

\subsubsection{Statement of the problem}
We work in $\O=(0;\pi)$. 
We consider a $\mathscr C^2$ function $f:[0;1]\rightarrow \R$, and the associated parabolic equation
\begin{equation} \label{eq:rd_Gral}
	\left\{
		\begin{array}{ll}
      \partial_t u - \Delta u = f(u) &\hbox{ in } \R_+ \times \Omega,\\
      u(0,x) = u_{0}(x)  &\hbox{ in } \Omega,\\
      {\partial u \over \partial \nu}(t,x)=0 &\text{ in }\R_+\times \partial \O,		\end{array}  
	\right.
\end{equation}
where $u_0$ is an initial condition satisfying the constraint 
$$0\leq u_0\leq 1.$$
Since our initial motivation, as explained in the first paragraph of this introduction, is to maximise the proportion of a subgroup of a population, such an $L^\infty$ constraint is natural. At the mathematical level, it should be noted that we could carry out the same analysis with any $L^\infty$ constraint of the form $0\leq u_0\leq \kappa$ by a simple change of variable. 

We define, for any $T>0$, the functional 
\begin{equation} \label{def:Operator}
  \mathcal{J}_{T}(u_{0}):=\int_{\Omega}u(T,x)dx.
\end{equation} 

The goal is to maximise $\mathcal J_T$ with respect to $u_0$. Since we are then again wondering how to maximise the proportion of a subgroup by controlling its distribution at the initial time, it is natural to introduce a $L^1$ constraint on $u_0$. This constraint is encoded by a parameter $m\in (0;|\O|)$ which is henceforth fixed. 

These considerations lead us to defining our admissible class as 
\begin{equation}
\mathcal A:=\left\{u_0\in L^\infty(\O)\,, 0\leq u_0\leq 1\text{ a.e. }\,, \int_\O u_0=m\right\},\end{equation}
and the variational problem under scrutiny throughout this paper is
\begin{equation}\label{Eq:Pv}\tag{$\bold{P}_f$}
\underset{u_0\in \mathcal A}\max  {\, \mathcal J_T}(u_0).\end{equation}
This problem $\eqref{Eq:Pv}$ was directly addressed by two of the authors in \cite{NadinToledo}, where expressions for the first and second order optimality conditions were provided. We need to recall them, to motivate and contextualise our results: if we consider $u_0\in \mathcal A$ and an admissible perturbation $h_0$ at $u_0$ (by ``admissible perturbation'' we refer to the fact that $h_0$ belongs to the tangent cone to the set $\mathcal A$ at $u_0$.
This tangent cone is the set of functions $h\in L^\infty(\O)$ such that, for any sequence of positive real numbers $(\varepsilon_n)_{n\in \N}$ decreasing to $0$, there exists a sequence of functions $(h_n)_{n\in \N}\in L^\infty(\O)^\N$ converging to $h$ as $n\rightarrow +\infty$, and $u_0+\varepsilon_nh_n\in\mathcal A$ for every $n\in\N$) then the first order G\^ateaux-derivative of $\mathcal J_T$ at $u_0$ in the direction $h_0$ is 
\begin{eqnarray}\label{Derivative:J_T}
  \langle \nabla \mathcal{J}_T(u_0),h_0\rangle  & = & \int_{\Omega} h_0(x)p(0,x)dx\end{eqnarray} where $p$ solves the adjoint equation
  \begin{equation} \label{eq:rd_p}
  \left\{
    \begin{array}{ll}
      -\partial_t p - \Delta p = f'({u})p \quad &\hbox{ in } (0,T) \times \Omega,\\
      \\
      p(T,x) = 1 \quad &\hbox{ in } \Omega,\\
      \\
      {\partial p \over \partial \nu}(t,x)=0 \quad &\hbox{ for all } t\in (0,T), \hbox{ for all } x\in \partial \Omega.
    \end{array} 
  \right.
\end{equation} 
Here $u$ is the solution of \eqref{eq:rd_Gral} with initial condition $u_0$.

The main result in \cite{NadinToledo} states the following:

\begin{theorem*} \label{Thm1} \cite{NadinToledo}
  There exist a solution $\overline{u}_0 \in \mathcal{A}$ of (\ref{Eq:Pv}).	Moreover, setting $\overline{u}$ as the solution of (\ref{eq:rd_Gral}) associated with this optimal initial data and $\overline{p}$ as the unique solution of (\ref{eq:rd_p}) for $u=\overline{u}$, there exists a non-negative real value $\overline{c}$ such that 
	\begin{enumerate}
		\item[i)] if $0 < \overline{u}_0(x)< 1$ then $\overline{p}(0,x)=\overline{c}$,
		\item[ii)]  if $\overline{p}(0,x)> \overline{c}$, then $ \overline{u}_0(x) = 1$,
		\item[iii)] if $\overline{p}(0,x)< \overline{c}$, then $ \overline{u}_0(x) = 0$.
	\end{enumerate}
Finally,  for almost every $x\in \{ \overline{p}(0,\cdot)=\overline{c}\}$, one has 
\begin{equation} \label{eq:f'} f'\big( \overline{u}_0 (x)\big) =- \overline{p}_t (0,x) / \overline{p}(0,x)\end{equation}
and the left-hand side belongs to $L^p_{loc} (\Omega)$. 
\end{theorem*}

The characterisation of $\overline{u}_0$ with the help of $p$ is almost complete here, except on the singular arc $\omega=\{0<u_0<1\}$. Note first that this singular arc might have a positive measure. It was even proved in \cite{NadinToledo} that, if $f$ is concave, then $\omega \equiv \Omega$. 
If $f'$ is monotonic, equation (\ref{eq:f'}) admits a unique solution and thus fully characterizes $\overline{u}_{0}$. But for a bistable nonlinearity $f_\theta(u)=u(1-u)(u-\theta)$, equation (\ref{eq:f'}) might have two solutions, one belonging to $[0;\eta]$ and the other to $(\eta;1]$, where $\eta=\eta(\theta)\in (0;1)$ is the unique real number such that $f_\theta$ is convex in $[0;\eta)$ and concave in $(\eta;1]$. It is then necessary to distinguish between these two possible roots in order to completely characterize $\overline{u}_{0}$ with $p$. 

  From the numerical point of view, the characterisation given by this result naturally leads to a gradient descent algorithm, which is not well-posed if we are not able to characterise $\overline{u}_{0}$ on $\omega$. Let us briefly describe this algorithm, which we detail further in section \ref{sec:numerics} of the present paper, to explain the core difficulty and how our theoretical results enable us to bypass it: starting from an initial configuration $u_0^0$, we seek to improve it to obtain a better admissible candidate $u_0^1$. We first compute the adjoint state $p_0^0$ associated with $u_0^0$. The problem arises if $p_0^0$  has "flat zones", in other words if there exists $c_0$ (necessarily unique) such that
$$\left\{p_0^0>c_0\}\right|<m\,, \left|\{p_0^0\geq c_0\}\right|>m\,, \left|\{p_0^0=c_0\}\right|>0.$$
On the set $\{p_0^0>c_0\}$, we replace $u_0^0$ by 1, while on $\{p_0^0<c_0\}$ we replace $u_0^0$ by zero. On $\{p_0^0=c_0\}$ we must replace $u_0^0$ with a root of \eqref{eq:f'}.  \eqref{eq:f'} can have two roots $\mu_1^0$ and $\mu_2^0$. These roots can be distinguished by the convexity of $f$: up to a relabelling, $f''(\mu_1^0)> 0$ and $f''(\mu_2^0)\leq0$. In \cite{NadinToledo}, the two possibilities were explored successively, which led to high computational costs. This was the main limitation of the numerical approach of \cite{NadinToledo}. Theorem \ref{Prop1} of the present paper shows that one should choose $\mu_2^0$. This significantly improves the running time of our algorithm and we refer to section \ref{sec:numerics} for examples.
\color{black}


\subsubsection{Related works}
A related problem has first been addressed by Garnier, Hamel and Roques in \cite{GHR}, where the authors consider a bistable reaction term $f(u):= u(1-u)(u-\theta)$, with $\theta\in (0,1)$, over the full line $\Omega=\R$. In this earlier paper, the authors did not investigate $\eqref{Eq:Pv}$, but they tried to optimize the initial datum in order to ensure the convergence to $u\equiv 1$ when $t\to +\infty$. They investigated numerically the particular case $u_{0}:= \mathds{1}_{(-\alpha-\frac{m}2,-\alpha)}+\mathds{1}_{(\alpha,\alpha+\frac{m}2)}$, and proved that in some situations the initial datum associated with $\alpha=0$ might lead to extinction (that is, $u(t,x)\to 0$ as $t\to +\infty$), while a positive $\alpha>0$ might lead to persistence (that is, $u(t,x)\to 1$ as $t\to +\infty$). Also, numerics for more general classes of initial datum indicate that fragmentation might favor species persistence. Hence, even if the problem we consider here is a bit different, we expect the maximiser to be fragmented, that is, non-smooth, for bistable nonlinearities.

More recently, this problem was also addressed in \cite{Elie}, in a slightly more general form, and for the criterion $\int_{\Omega}|1-u(T,x)|^{2}dx$. These authors investigated in particular various conditions ensuring that the maximiser $\overline{u}_{0}$ is constant with respect to $x$, and, reciprocally, that the constant initial datum is a local maximiser. 

\bigskip

\subsection{Main results of the paper}
The main contributions of this paper are the following:
\begin{itemize}
\item When $u_0$ does not saturate the $L^\infty$ constraints (i.e. when the set $\omega:=\{0<u_0<1\}$ has positive measure), we prove in theorem \ref{Prop1}  that any maximiser $\overline u_0$ must necessarily be in a zone of concavity of $f$: in $\omega$, $f''(\overline u_0)\leq 0$.
\item When $f$ is convex, theorem \ref{Th:Convex} characterizes explicitely a global maximiser, using rearrangement techniques. Here, the presence of Neumann boundary conditions prohibits using in a straightforward manner the results of \cite{Bandle}, and we need to adapt some points of the proof to this case.
\item When $f$ is a bistable non-linearity, we improve the algorithm initially introduced in \cite{NadinToledo} and display several numerical simulations. One-dimensional simulations are displayed that exemplify the fact that theorem \ref{Prop1} significantly improves the computational time of optimisation algorithms. We also provide two-dimensional simulations.
\end{itemize}

\subsubsection{Characterisation of the singular arc}

Let us first recall the expression of the second order derivative \cite{NadinToledo}:
   
  \begin{eqnarray}
  \left\langle \nabla^2 \mathcal{J}_T(u_0),h_0 \right\rangle & = &\iint_{(0;T)\times \O}f''(u(t,x))p(t,x)h^2(t,x) \ dx \ dt,
\end{eqnarray}
where $p$ solves \eqref{eq:rd_p} and $h$ solves 
\begin{equation}\label{Eq:h}\begin{cases}
\frac{\partial h}{\partial t}-\Delta h=f'(u)h\text{ in }(0;T)\times \O\,, 
\\ h(0,x)=h_0(x)\,, 
\\ \frac{\partial h}{\partial \nu}=0\text{ on }(0;T)\times \partial \O.\end{cases}\end{equation}

We can now state our main result:
\begin{theorem} \label{Prop1}
Assume $\O=(0;\pi)$.
Let  $\overline{u}_0$ be a solution of  \eqref{Eq:Pv}. If the set $\Omega_{\overline{c}}:=\{x\in \Omega: 0<\overline{u}_{0}(x)<1\}$ has a positive measure then, for   almost every \color{black} interior point $x$ of $\Omega_{\overline{c}}$, there holds
\begin{equation}f''(\overline u_0(x))\leq 0.\end{equation}
\end{theorem}

\begin{remark}
The method we put forth is reminiscent of one that was used in \cite{NadinToledo} to study the case of a constant initial condition and to prove that such a constant $u_0$ was always a maximiser in the case of monostable non-linearities. Here, working with interior point greatly complexifies the situation and calls for two-scale asymptotic expansions.

The main drawback to our approach is that it can not cover the case of singular (e.g. Cantor-like) singular arcs, and it is a very interesting question to prove that such a property holds for any point of the singular arc $\Omega_{\overline c}$. We comment on the main difficulties of this approach in the conclusion and only state, for the moment, that the main problem is related to the ubiquitous problem of separation of phases in homogenisation \cite{AllaireBriane}.\end{remark}

\subsubsection{Convex non-linearities and rearrangements}
We present, in this section, a characterisation of maximisers when the non-linearity $f$ is convex, using rearrangement and symmetrisation techniques. It should be noted that, since we are working with Neumann boundary conditions, it is not possible to use directly well-known parabolic isoperimetric inequalities \cite{Bandle,RakotosonMossino,Alvino1990}. We refer to \cite{M2021,Bandle,Rakotoson} and the references therein for an introduction to parabolic isoperimetric inequalities, and only underline here that  the most precise results available in the literature only encompass the case of Dirichlet boundary conditions. For Neumann boundary conditions, a large literature \cite{Bramanti1991,Ferone2005,Maderna1979} is devoted to such questions. Usually, it involves a comparison of the solution with the solution of a Dirichlet, or of a mixed Dirichlet-Neumann problem, and makes use of constants appearing in relative isoperimetric inequalities. It is unclear whether these comparison results could be used in our case. Since we are working in the one-dimensional case, a direct adaptation of the proof of \cite{Bandle} yields the required results.

 In order to state our result, let us introduce the following notation: let $\tilde u_0$ be defined as 
\begin{equation}\tilde u_0:=\mathds 1_{(0;m)}\in \mathcal A.\end{equation}

\begin{theorem}\label{Th:Convex}
Assume $f$ is a convex, $\mathscr C^1$ function such that $f(0)=0$. Then $\tilde u_0$ is a solution of \eqref{Eq:Pv}.
\end{theorem}

\begin{remark}
\begin{itemize}
\item It should be noted that $\tilde u_0$ appears as the solution of many other optimisation problems in population dynamics, most specifically for the monostable case \cite{BHR,KaoLouYanagida,MNP} with Neumann boundary conditions.
\item The maximiser $\mathds 1_{(0;m)}$ is clearly not unique, since $\mathds 1_{(1-m;1)}$ is also a maximiser for example. This provides an example of non-uniqueness of the maximiser. 

\item As a corollary, $\tilde u_0:=\mathds 1_{(0;m)}$ is a minimiser of $\mathcal{J}_{T}$ over $\mathcal{A}$ if $f$ is concave. 
\end{itemize}
\end{remark}

\section{Proof of Theorem \ref{Prop1}}

\subsection{Notations, plan of the proof and first simplification}
\paragraph{Second order optimality conditions}
We recall the expression of the second order derivative of $\mathcal J$: for an admissible perturbation $h_0$, we have
$$\left\langle \nabla^2 \mathcal{J}_T(u_0),h_0 \right\rangle  =  \int_{0}^{T} \int_{\Omega}f''(u(t,x))p(t,x)h^2(t,x) \ dx \ dt,$$ where $p$ solves \eqref{eq:rd_p} and $h$ solves \eqref{Eq:h}. 
 
Let $\overline u_0$ be a solution of \eqref{Eq:Pv}. We assume that the set $\Omega_{\overline c}:=\{0<\overline u_0<1\}$ has a non-empty interior and we want to prove that $f''(\overline u_0)\leq 0$ almost everywhere on the interior of this set.
To do this, we need the following expression of second order optimality conditions:

\begin{lemma}\label{Le:OOC}For every  $h_{0}\in L^{\infty}(\Omega)$ supported in $\Omega_{\overline c}$, such that $\int_{\Omega}h_{0}=0$, there holds
\begin{equation}\label{Eq:Rev0} \iint_\OT f''(u(t,x))p(t,x)h^2(t,x) \ dx \ dt\leq 0\end{equation} where $h$ is the solution of \eqref{Eq:h} associated with the initial condition $h_0$.\end{lemma}
\begin{proof}[Proof of Lemma \ref{Le:OOC}]
We first notice the following thing: let, for any $n\in \N^*$, the set $F_n$ be defined as 
$$F_n:=\left\{\frac1n<u_0<1-\frac1n\right\}.$$ Then, for any $L^\infty$ function $h_0$ supported in $F_n$ (in the sense that $h_0\mathds 1_{F_n}=h_0$) such that $\int_\O h_0=0$, if $h$ is the solution of \eqref{Eq:h} associated with $h_0$, we have
$$ \iint_\OT f''(u(t,x))p(t,x)h^2(t,x) \ dx \ dt\leq 0.$$ This is a consequence of the fact that, for any $\tau$ such that $|\tau|$ is small enough, $\overline u_0+\tau h_0$ is an admissible initial condition.

Let us now consider $h_0\in L^\infty(\O_{\bar c})$ satisfying $\int_\O h_0=0$. We define, for any $n\in \N$, 
$$h_{0,n}:=\mathds 1_{F_n}\left(h_0-\frac{\int_{F_n}h_0}{|F_n|}\right).$$ For every $n\in \N$, $h_{0,n}$ is supported in $F_n$ and verifies $\int_{\O}h_{0,n}=0$.  Hence, defining, for any $n\in \N$, $h_n$ as the solution of \eqref{Eq:h} associated with the initial condition $h_{0,n}$ we have 
\begin{equation}\label{HU} \iint_\OT f''(u(t,x))p(t,x)h_n^2(t,x) \ dx \ dt\leq 0.\end{equation}
However, there holds
$$h_{0,n}\underset{n\to \infty}\rightarrow h_0\text{ in $L^2(\O)$},$$ which, by standard parabolic estimates, entails
$$h_n\underset{n\to \infty}\rightarrow h\text{ in $L^2(\OT)$}$$ where $h$ is the solution of \eqref{Eq:h} associated with the initial condition $h_0$. Passing to the limit $n\to \infty$ in \eqref{HU} yields the conclusion.
\end{proof}

Since we want to retrieve, from the second order optimality conditions \eqref{Eq:Rev0}, an information of $f''(\overline u_0)$, we need to find a perturbation $h_0$ such that the ensuing solution $h$ satisfies, roughly speaking, 
$$h^2(t,x)\approx C \delta_{t=t_0} g(x),$$  for a certain function $g$. As $h$ is a solution of a parabolic equation, one possibility to obtain such a behaviour is to choose a highly oscillating initial condition, say $h_0(x)=\cos(kx)$ with a large integer $k$. This would give $h^2(t,x)\approx e^{-tk^2}\cos(kx)^2$, which, thanks to the Laplace method, does concentrate around $t=0$ up to a proper rescaling. This however is not particularly convenient, as such a perturbation is not admissible: it is not supported in $\O_{\bar c}$. To overcome this difficulty, we need to truncate such highly oscillating perturbations, thus choosing a perturbation of the form $\theta(x)\cos(kx)$, with $\theta$ a cut-off function, leading to two-scale asymptotic expansions. The objective is to pick the correct function $\theta$.

In order to summarise our approach, let us fix notations: we pick an optimiser $\overline u_0$, we define $\O_{\bar c}:=\{0<\overline u_0<1\}$, and we set $\mathring{\O}_{\bar c}$ as the interior of $\O_{\bar c}$.  To prove Theorem \ref{Prop1} we argue by contradiction: assume that, for some $\delta>0$, 
\begin{equation}\label{Eq:Rut}\left|\{f''(\overline u_0)\geq \delta\}\cap \mathring{\O}_{\bar c}\right|>0.\end{equation} 
Since $\mathring{\O}_{\bar c}$ is an open set, we can write it as a union of intervals
\begin{equation}\mathring{\O}_{\bar c}=\bigcup_{k=0}^\infty (a_k;b_k).\end{equation} By \eqref{Eq:Rut}, there exists $n_0\in \N$ such that
\begin{equation*}\left|\{f''(\overline u_0)\geq \delta\}\cap (a_{n_0};b_{n_0})\right|>0,\end{equation*} so that there exists $\epsilon>0$ such that, for the same $n_0$, we have 
\begin{equation}\label{Eq:Rut2}\left|\{f''(\overline u_0)\geq \delta\}\cap (a_{n_0}+\epsilon;b_{n_0}-\epsilon)\right|>0.\end{equation} We fix such an $\epsilon>0$.

To alleviate notations, define $E:=\{f''(\overline u_0)\geq \delta\}\cap (a_{n_0}+\epsilon;b_{n_0}-\epsilon)$. As $p(0,\cdot)>0$ by the parabolic maximum principle, \eqref{Eq:Rut2} yields
\begin{equation}\label{Eq:TK}\int_{\O} f''(\overline u_0(\cdot))p(0,\cdot)\mathds 1_E>0.\end{equation}
We  approximate in $L^1(a_{n_0};b_{n_0})$ the function $\mathds 1_E$ by a sequence $\{\psi_k\}_{k\in \N}$ of uniformly bounded, non-negative, $\mathscr C^\infty$ functions that are compactly supported in $(a_{n_0};b_{n_0})\subset \mathring{\O}_{\bar c}$. In particular, the sequence $\{\psi_k^2\}_{k\in \N}$ also converges to $\mathds 1_E$ in $L^1(\O)$, so that \eqref{Eq:TK} implies that for $K$ large enough 
\begin{equation}\label{Eq:TK2}\int_{\O} f''(\overline u_0(\cdot))p(0,\cdot)\psi_K^2>0.\end{equation}
We fix $K$ large enough so that \eqref{Eq:TK2} holds and we set, for this index $K$, 
$$\theta:= \psi_K\in \mathscr C^\infty(\O).$$
The sequence of truncated, highly oscillating initial conditions is 
$$\overline h_{k,0}:=\theta(\cdot)\left(\cos(k\cdot)+\alpha_k\right)$$ where \begin{equation}\label{Eq:alpha}\alpha_k=-\frac{\int_{\O} \theta \cos(k\cdot)}{\int_\O \theta}\end{equation} simply ensures that $\int_\O \overline{h}_{k,0}=0$. This constant does not play a role in the upcoming analysis for the following reason:
\begin{enumerate}
\item First, by setting $h_{k,0}:=\theta(x)\cos(kx)$ and by defining $h_k$ as the solution of \eqref{Eq:h} associated with the initial condition $h_k^0$ we shall show that 
\begin{equation}\label{Eq:Climb}
\iint_{(0;T)\times \O} f''(u)ph_k^2\underset{k\to \infty}\sim\frac{C}{k^2}\int_\O f''(\overline u_0(\cdot))p(0,\cdot)\theta^2(\cdot)
\end{equation}
for some constant $C$. This is the core of the proof, and will take up the remainder of this section of the paper.

It is also immediate by parabolic regularity to obtain that the sequence $\{h_k\}_{k\in \N}$ is uniformly bounded in $L^2((0;T)\times \O)$.
\item Second we observe that, as $\theta\in \mathscr C^4$, the Riemann-Lebesgue lemma in particular ensure that $\alpha_k=\underset{k\to \infty}{\mathscr O}\left(\frac1{k^4}\right)$.
\item If we now set $z$ as the solution of \eqref{Eq:h} associated with the initial condition $\theta(\cdot)$, the solution $\overline h_k$ associated with the (admissible) initial condition $\overline h_{k,0}$ is given by $\overline h_{k}=h_k+\alpha_k z$. Then the second order derivative in the admissible direction $\overline h_{k,0}$ is given by
\begin{align*}
\left\langle \nabla^2 \mathcal{J}_T(u_0),\overline h_{k,0} \right\rangle  &=  \iint_\OT f''(u(t,x))p(t,x)\left(\overline h_k\right)^2(t,x) \ dx \ dt
\\&=  \iint_\OT f''(u(t,x))p(t,x) h_k^2(t,x) \ dx \ dt
\\&+2\alpha_k\iint_\OT f''(u(t,x))p(t,x)z(t,x)h_k(t,x)\ dx \ dt
\\&+\alpha_k^2\iint_\OT f''(u(t,x))p(t,x)z^2(t,x)\ dx \ dt
\end{align*}
Taking into account \eqref{Eq:Climb} and the fact that $\alpha_k=\underset{k\to \infty}{\mathscr O}\left(\frac1{k^4}\right)$ leads to 
$$\left\langle \nabla^2 \mathcal{J}_T(u_0),\overline h_{k,0} \right\rangle\underset{k\to \infty}\sim\frac{C}{k^2}\int_\O f''(\overline u_0(\cdot))p(0,\cdot)\theta^2(\cdot)>0,$$ a contradiction.
\item As a consequence, the theorem is proved, provided we can prove \eqref{Eq:Climb}, and we henceforth focus on this point. 
\end{enumerate}

%
%
%
%

\color{black}
%
%
%
 \subsection{Asymptotic expansion of $h_k$}  
  \color{black}
Let $\theta$ be given as above.
We consider the following sequence of equations: let, for any $k\in \N$, $h_k$ be the solution of
\begin{equation}\label{Eq:1}
\begin{cases}
\displaystyle\frac{\partial h_k}{\partial t}-\frac{\partial^2h_k}{\partial x^2}=f'({u})h_k\,, 
\\
\\\displaystyle\frac{\partial h_k}{\partial \nu}=0\,, 
\\
\\ h_k(0,x)=h_{k,0}(x)=\theta(x)\cos(kx).\end{cases}\end{equation}
In this context, it is natural \cite{Allaire} to look for a two-scale asymptotic expansion of $h_k$ of the form 
\begin{equation}\label{Eq:AsymptoticDevelopment}h_k(t,x)\approx h_k^0(k^2t,x,kx)+\frac1kh_k^1(k^2t,x,kx)+\dots\end{equation}
which, after a formal identification at the first and second order, gives the following equations on $h_k^0$ and $h_k^1$:
\begin{equation}\label{Eq:h0}
\begin{cases}
\displaystyle\frac{\partial h_k^0}{\partial s}-\frac{\partial^2 h_k^0}{\partial y^2}=0\,, 
\\
\\\displaystyle\frac{\partial h_k^0}{\partial \nu}=0\,, 
\\
\\h_k^0(0,x,y)=\theta(x)\cos(y).\end{cases}\end{equation}
and 
\begin{equation}\label{Eq:h1}
\begin{cases}
\displaystyle\frac{\partial h_k^1}{\partial s}-\frac{\partial^2 h_k^1}{\partial y^2}=2\frac{\partial^2h_k^0}{\partial x\partial y}, 
\\
\\\displaystyle\frac{\partial h_k^1}{\partial \nu}=0\,, 
\\
\\h_k^1(0,x,y)=0.\end{cases}\end{equation}
Equation \eqref{Eq:h0} can be solved explicitly, giving 
\begin{equation}\label{Eq:ExpH0}h_k^0(s,x,y)=\theta(x)\cos(y)e^{-s}.\end{equation}
This, in turn, allows to solve equation \eqref{Eq:h1} as
\begin{equation}\label{Eq:ExpH1}h_k^1(s,x,y)=-2se^{-s}\theta'(x)\sin(y).\end{equation}

\begin{proposition}\label{Pr:AsymptoticDevelopment}
The asymptotic expansion \eqref{Eq:AsymptoticDevelopment} is valid in $L^2(\O)$ in the following sense: there exists $M>0$   that depends on the time horizon $T$ \color{black} such that, if we define
$$R_k:=h_k(t,x)-h_k^0(k^2t,x,kx)-\frac1k h_k^1(k^2t,x,kx)$$ then, for any $t\in (0;T)$, 

\begin{equation}\label{Eq:ControleReste}
\Vert R_k(t,\cdot)\Vert_{L^2(\O)}\leq \frac{M}{k^2}. \end{equation}In particular,
\begin{equation}\label{Durin}\iint_{(0;T)\times \O}R_k^2\leq \frac{M^2}{k^4}\,, \int_0^T \Vert R_k\Vert_{L^2(\O)}\leq \frac{MT}{k^2}.
\end{equation}
\end{proposition}

\begin{proof}[Proof of Proposition \ref{Pr:AsymptoticDevelopment}]
To prove this Proposition, we write down the equation satisfied by $R_k$. Straightforward computations show that $R_k$ solves
\begin{equation}\label{Eq:Rk}
\partial_t R_k-\Delta R_k-f'(u)R_k:=f'(u)\left(h_k^0+\frac1kh_k^1\right)+\frac{\partial^2 h_k^0}{\partial x^2}+2\frac{\partial^2 h_k^1}{\partial x\partial y}+\frac1k\frac{\partial^2h_k^1}{\partial x^2},
\end{equation}
and all the functions on the right hand side are evaluated at $(k^2t,x,kx)$ (we dropped this for notational convenience). We now introduce the following notations:
$$
\begin{cases}
W_0:=f'(u)\,, 
\\
\\V_{0,k}(t,x):=h_k^0(k^2t,x,kx)+\frac1kh_k^1(k^2t,x,kx)\,, 
\\
\\ V_{1,k}:=-\frac{\partial^2 h_k^0}{\partial x^2}(k^2t,x,kx)\,, 
\\
\\ V_{2,k}:=-2\frac{\partial^2 h_k^1}{\partial x \partial y}-\frac1k\frac{\partial^2 h_k^1}{\partial x^2}.
\end{cases}$$
First of all, since $0\leq u\leq 1$ and $f\in \mathscr C^1$, there exists $M_0>0$ such that

\begin{equation}\label{Eq:W0}
\Vert W_0\Vert_{L^\infty((0;T)\times \O)}\leq M_0.
\end{equation}

We gather the main estimates on source terms in the following Lemma:
\begin{lemma}\label{Le:Estimates}
There exists $\tilde M>0$ such that
\begin{equation}\label{Eq:V0k}
\int_0^T\Vert V_{0,k}(t,\cdot)\Vert_{L^2( \O)}\leq \frac{\tilde M}{k^2},
\end{equation}
\begin{equation}\label{Eq:V1k}
\int_0^T\Vert V_{1,k}(t,\cdot)\Vert_{L^2 \O)}dt\leq \frac{\tilde M}{k^2}.
\end{equation}

\begin{equation}\label{Eq:V2k}
\int_0^T\Vert V_{2,k}(t,\cdot)\Vert_{L^2 \O)}dt\leq \frac{\tilde M}{k^2}.
\end{equation}

\end{lemma}

\begin{proof}[Proof of Lemma \ref{Le:Estimates}]
We prove the three estimates separately. Let us recall the following consequence of the Laplace method: for any integer $m\in \mathbb{N}^*$, one has
\begin{equation}\tag{$\bold{I}_m$}\label{Eq:LaplaceLambda}
\int_0^T t^{m-1}e^{-k^2t}dt\underset{k\to \infty}\sim\frac{(m-1)!}{k^{2m}}.\end{equation}
\textbf{Proof of \eqref{Eq:V0k}}

By the triangle inequality we get, for any $t\in (0;T)$, 
$$\Vert V_{0,k}(t,\cdot)\Vert_{L^2}\leq \Vert h_k^0(k^2t,\cdot,k\cdot)\Vert_{L^2(\O)}+\frac1k\Vert h_k^1(k^2t,\cdot,k\cdot)\Vert_{L^2(\O)}$$ We first use the explicit expressions \eqref{Eq:ExpH0}-\eqref{Eq:ExpH1} for $h_k^0$ and $h_k^1$ to obtain, using $\Vert \theta\Vert_{L^\infty}\leq 1$,
\begin{equation}\label{Borne:h_k^0_L2}
  \begin{split}
    \Vert h_k^0(k^2t,\cdot,k\cdot)\Vert_{L^2(\O)}^2 &= \int_\O \theta(x)^2\cos(kx)^2e^{-2k^2 t} d x\color{black} \\
    & \leq e^{-2k^2t}|\O|,
  \end{split}
\end{equation}
and integrating this inequality between $0$ and $T$ gives
$$\int_0^T\Vert h_k^0(k^2t,x,k\cdot)\Vert_{L^2(\O)}dt\leq \frac{\tilde M}{k^2}.$$
In the same way, we have, for any $t\in (0;T)$,
\begin{equation}\label{Borne:h_k^1_L2}
  \begin{split}
    \Vert h_k^1(k^2t,\cdot,k\cdot)\Vert_{L^2(\O)}^2 &=4 k^{4}t^2e^{-2k^2t}\int_\O (\theta'(x))^2\sin(x)^2dx \\
    & \leq C k^4t^2e^{-2k^2t}|\O|\cdot\Vert \theta'\Vert_{L^\infty}^2,
  \end{split}
\end{equation} for some constant $C$.
Taking the square root and integrating in time we get, for a constant $C'$,

$$\frac1k\int_0^T\Vert h_k^1(k^2t,\cdot,k\cdot)\Vert_{L^2(\O)}dt\leq C' |\O|\cdot\Vert \theta'\Vert_{L^\infty} k\int_0^Tte^{-k^2t}dt.$$

Using \eqref{Eq:LaplaceLambda} with $m=2$ gives

$$
\int_0^T\frac1k\Vert h_k^1(k^2t,\cdot,k\cdot)\Vert_{L^2(\O)}dt\leq \frac{M_1}{k^3}$$ for some constant $M_1$.

Summing these two contributions gives \eqref{Eq:V0k}.

\textbf{Proof of \eqref{Eq:V1k}}
This follows from the same arguments, by simply observing that 
$$V_{1,k}(t,x)=-e^{-k^2t}\theta''(x)\cos(kx).$$

\textbf{Proof of \eqref{Eq:V2k}} We once again split the expression and estimate separately 
$$\int_0^T \left\Vert \frac{\partial^2 h_k^1(k^2t,\cdot,k\cdot)}{\partial x \partial y}\right\Vert_{L^2(\O)}dt \text{ and }\frac1k\int_0^T \left\Vert \frac{\partial^2 h_k^1(k^2t,\cdot,k\cdot)}{\partial x^2}\right\Vert_{L^2(\O)}dt.$$

We first observe that for any $t\in (0;T)$, we have
 $$\frac{\partial^2 h_k^1(k^2t,\cdot,k\cdot)}{\partial x \partial y}=-4k^2te^{-k^2t}\theta''(x)\cos(y).$$In particular, for any $t\in (0;T)$
 $$ \left\Vert \frac{\partial^2 h_k^1(k^2t,\cdot,k\cdot)}{\partial x \partial y}\right\Vert_{L^2(\O)}\leq 4k^2te^{-k^2t}|\O|\cdot\Vert \theta''\Vert_{L^\infty}$$ so that the Laplace method \eqref{Eq:LaplaceLambda} with $\lambda=2$ gives the bound
 $$\int_0^T \left\Vert \frac{\partial^2 h_k^1(k^2t,\cdot,k\cdot)}{\partial x \partial y}\right\Vert_{L^2(\O)}dt\leq\frac{M_2}{k^2} $$ for some constant $M_2$. The proof of the control of the second term follows along exactly the same lines.
 
\end{proof}
Let us now prove estimate \eqref{Eq:ControleReste}. The equation on $R_k$ rewrites
\begin{equation}\label{Eq:Reste}
\partial_t R_k-\Delta R_k-W_0R_k=W_0 V_{0,k}+V_{1,k}+V_{2,k}.\end{equation}
Multiplying the equation by $R_k$ and integrating by parts in space gives

$$\frac12\partial_t \int_\O R_k^2+\int_\O|\n R_k|^2-\int_\O W_0 R_k^2\leq \Vert R_k\Vert_{L^2(\O)}\left(M_0\Vert V_{0,k}(t,\cdot)\Vert_{L^2(\O)}+\Vert V_{2,k}(t,\cdot)\Vert_{L^2(\O)}+\Vert V_{2,k}(t,\cdot)\Vert_{L^2(\O)}\right).$$
In other words, bounding $W_0$ by $M_0$ and defining $g(t):=\Vert R_k(t,\cdot)\Vert_{L^2(\O)}^2$ we obtain
$$\frac12 g'(t)\leq M_0g(t)+\sqrt{g(t)}\left(M_0\Vert V_{0,k}(t,\cdot)\Vert_{L^2(\O)}+\Vert V_{2,k}(t,\cdot)\Vert_{L^2(\O)}+\Vert V_{2,k}(t,\cdot)\Vert_{L^2(\O)}\right).$$ Furthermore, $R_k(0,\cdot)=0$. We thus obtain, by the Gronwall Lemma, for any $t\in (0;T)$,
$$\sqrt{g(t)}e^{-M_0t}\leq \int_0^t e^{-M_0s}\left(M_0\Vert V_{0,k}(t,\cdot)\Vert_{L^2(\O)}+\Vert V_{2,k}(t,\cdot)\Vert_{L^2(\O)}+\Vert V_{2,k}(t,\cdot)\Vert_{L^2(\O)}\right).$$ Hence, by Lemma \ref{Le:Estimates} we get for some constant $N_0$ and any $t\in( 0;T)$,

$$\Vert R_k(t,\cdot)\Vert_{L^2(\O)}\leq N_0\int_0^T\left(\Vert V_{0,k}(t,\cdot)\Vert_{L^2(\O)}+\Vert V_{2,k}(t,\cdot)\Vert_{L^2(\O)}+\Vert V_{2,k}(t,\cdot)\Vert_{L^2(\O)}\right)\leq \frac{N_0\tilde M}{k^2}.$$

\end{proof}
 
\subsection{Back to the proof}
We turn back to the proof of Theorem \ref{Prop1} and, more precisely, to the proof of \eqref{Eq:Climb}.
\begin{proof}[Proof of Theorem \ref{Prop1}]
We use the same  $\theta$ as above and the same notation $\alpha_k$ as in the introduction of the proof (equation \eqref{Eq:alpha}).
Let us now consider the initial perturbation $\overline h_{k,0}:=\theta(x)(\cos(kx)+\alpha_k)$. We recall that $z$ is the solution of \eqref{Eq:h} with initial condition $\theta$ and that $h_k$ is the solution of \eqref{Eq:h} with initial condition $\theta(\cdot)\cos(k\cdot)$. Hence, since the equation is linear, 
$$\overline h_k=h_k+\alpha_k z.$$ By parabolic regularity, 
\begin{equation}\label{Eq:Z}
\sup\Big(\sup_{k\in \N} \Vert h_k\Vert_{L^2(\OT)}\,,\Vert z\Vert_{L^2(\OT)}\Big)<\infty.\end{equation} Since $\theta \in \mathscr C^4$, the Riemann-Lebesgue lemma implies
\begin{equation}\label{Eq:Es}
\alpha_k=\underset{k\to \infty}{\mathscr O}\left(\frac1{k^4}\right).\end{equation}
\color{black}

We define $$F(t,x):=f''(\overline{u}(t,x))\overline{p}(t,x)$$ so that the second order derivative of $\mathcal{J}_T$ in $\overline{u}_0$ rewrites as
\begin{equation}\label{Eq:D2J}
  \begin{split}
    \langle\n^2 \mathcal J_T(u_0),\overline h_{k,0}\rangle&= \iint_\OT F{\overline h_k}^2
    \\&=\iint_\OT F h_k^2
    \\&+2\alpha_k \iint_\OT F zh_k
    \\&+\alpha_k^2\iint_\OT Fz^2.
    \end{split}\end{equation}    
    
    We focus on the first term:
   \begin{equation}\begin{split}\iint_\OT F h_k^2
    & = \iint_\OT F(t,x){ \left( R_k(t,x) + V_{0,k}(t,x) \right) }^2 \ dx \ dt,\\
    & = \iint_\OT F(t,x)\left( {R_k(t,x)}^2 + 2R_k(x,t)V_{0,k}(t,x) + {V_{0,k}(t,x)}^2 \right) \ dx \ dt.
  \end{split}
\end{equation}
From the assumptions on $f$ and the estimates on $u$ and $p$, it easy to see that 
\begin{equation}\label{Borne:F}
  \Vert F \Vert_{L^\infty((0;T)\times \O)}\leq M_3.
\end{equation}
Gathering \eqref{Borne:F} and \eqref{Durin} it follows that 
\begin{equation}\label{Borne:D2J_term1}
  \iint_\OT F(t,x) {R_k(t,x)}^2 \ dx \ dt = \mathscr O(k^{-4})
\end{equation}
and similarly, gathering \eqref{Eq:V0k} and \eqref{Eq:ControleReste}  we obtain 
\begin{equation}\label{Borne:D2J_term2}
  \iint_\OT F(t,x) R_k(t,x) V_{0,k}(t,x) \ dx \ dt = \mathscr O(k^{-4})
\end{equation}
Let us now study the term 
\begin{align*}
  \iint_\OT F(t,x) {V_{0,k}(t,x)}^2 \ dx \ dt & = \iint_\OT F(t,x) \left( h_k^0(t,x) +\frac1 k h_k^1(t,x) \right)^2\\
  & = \iint_\OT F(t,x) \left( {h_k^0(t,x)}^2 + 2\frac1 k h_k^0(t,x)h_k^1(t,x) + \frac1 {k^2} {h_k^1(t,x)}^2 \right) dx dt
\end{align*}
Once again we split the expression. Applying  the Cauchy-Swchartz inequality, and  using  estimates  \eqref{Borne:h_k^0_L2}-\eqref{Borne:h_k^1_L2} it follows that the second term verifies  
\begin{equation}\label{Borne:D2J_term3_2}
  \begin{split}
    \Bigg | 2 \frac1 k  \iint_\OT F(t,x)h_k^0(t,x)h_k^1(t,x) dx dt \Bigg | & \leq 2 \frac{M_3}{k}  \int_0^T \Vert h_k^0 \Vert_{L^2(\O)}\Vert h_k^1 \Vert_{L^2(\O)} dt \\
  & \leq 2 M_4 k \int_0^T t e^{-2k^2t} dt,\\
  & = \mathscr O\left(\frac1 {k^3}\right).
  \end{split}
\end{equation}
The last step in the above expression follows directly from \eqref{Eq:LaplaceLambda} with $m=2$. 

We obtain in a similar way the following estimate on the third term:
\begin{equation}\label{Borne:D2J_term3_3}
  \begin{split}
    \Bigg | \frac1 {k^2} \iint_\OT F(t,x){h_k^1(t,x)}^2 dx dt \Bigg | & \leq \frac{M_3}{k^2}  \int_0^T \Vert h_k^1 \Vert_{L^2(\O)}^2 dt \\
  & \leq M_4 k^2 \int_0^T t^2 e^{-2k^2t} dt,\\
  & = \mathscr O\left(\frac1 {k^4}\right).
  \end{split}
\end{equation}
In this case we applied \eqref{Eq:LaplaceLambda} for $m=3$. 

Finally, let us study the first term which can be written as
\begin{equation}
  \begin{split}
    \iint_\OT F(t,x){h_k^0(t,x)}^2 dx dt & = \int_0^T e^{-2k^2t} G(t) \ dt
  \end{split}
\end{equation} 
where $G (t)\color{black}:=\int_{\O} F(t,x){\theta(x)}^2{\cos(kx)}^2 dx$ is a continous function of time as a consequence of parabolic regularity.

%
%
However, $\theta$ was chosen so that $$\int_\O f''(\overline u_0(\cdot))p(0,\cdot)\theta^2>0.$$ As $\cos(k\cdot)^2=\frac12\left(1+\cos(2\cdot)\right)\underset{k\to \infty}{\rightharpoonup} \frac1 2$, it follows that for any $k$ large enough $G(0)>0$. \color{black} Furthermore, from the Laplace method, when $k\to \infty$, one has 
\begin{equation}\label{Borne:D2J_term3_1}
  \begin{split}
    \iint_\OT F(t,x){h_k^0(t,x)}^2 dx dt \sim \frac{1}{2k^2} G(0).
  \end{split}
\end{equation}

Gathering the estimates in \eqref{Borne:D2J_term1},\eqref{Borne:D2J_term2}, \eqref{Borne:D2J_term3_2}, \eqref{Borne:D2J_term3_3}, \eqref{Borne:D2J_term3_1} and plugging them into the second derivative of the functional $\mathcal{J}_T$ given by \eqref{Eq:D2J} it follows that 
\begin{equation}\label{Eq:20}
  \iint_\OT F h_k^2\underset{k\to \infty}{\sim} \frac{1}{2k^2} G(0).
\end{equation}

We go back to \eqref{Eq:D2J}. By \eqref{Eq:Z}-\eqref{Eq:Es} and by \eqref{Eq:20} we have
\begin{equation}\begin{split}  \langle\n^2 \mathcal J_T(u_0),\overline h_{k,0}\rangle=\iint_\OT F h_k^2+\underset{k\to \infty}{\mathscr O}\left(\frac1{k^4}\right)
\underset{k\to \infty}{\sim} \frac{1}{2k^2} G(0)\end{split}
\end{equation}
%
This means that for $k$ suficiently large, 
\[k^2\langle\n^2 \mathcal J_T(u_0),\overline h_{k,0}\rangle> 0\]
which contradicts the fact that $\overline {u}_0$ is a maximiser of $\mathcal{J}_T$. The proof of the Theorem is complete.\end{proof}
 \color{black}
\section{Proof of Theorem \ref{Th:Convex}}
The proof follows essentially from the same arguments as in \cite{Bandle}. We thus only present the   main steps \color{black} that are in order so as to apply the methods of \cite{Bandle}.
We define $g(u):= f(u)+cu$, with $c>\|f'\|_{L^{\infty}}$, so that $g$ is increasing.

%

\paragraph{Reduction to a bang-bang maximiser.} We recall that bang-bang functions are defined as characteristic functions of subsets of $\O$, that is, functions only taking values 0 and 1.
First, as $f$ is convex, it follows from the same arguments as Proposition 6 of \cite{NadinToledo} that $\mathcal J_T$ is convex. Hence, one can restrict to maximisers among the extremal points of $\mathcal{A}$. These are exactly bang-bang function: $u_0$ satisfies $u_{0}=0$ or 1 almost everywhere on $(0,\pi)$.

\paragraph{Reduction to a periodic problem.}
Next, for all $t\in [0,T]$, we extend $u(t,\cdot)$ to $(-\pi;\pi)$ by symmetrisation with respect to $0$, and we then extend it to $\R$ by $2\pi$-periodicity. The Neumann boundary conditions at $x=0$ and $x=\pi$ ensure that the extended function is of class $\mathcal{C}^{1}$, and   it \color{black} thus satisfies the equation on the torus:
\begin{equation}\label{Eq:Periodic}
\begin{cases}
\frac{\partial u}{\partial t}-\Delta u+cu=g(u)\text{ in }(0;T)\times \mathbb T\,, 
\\ u(0,\cdot)=u_0.\end{cases}\end{equation}

Let us also recall some basic facts about rearrangements.
\paragraph{Periodic rearrangements }
We recall the definition of the periodic rearrangement: for any periodic function $u:\mathbb T\to \R_+$ if we identify $\mathbb T$ with $[-\pi;\pi]$ there exists a unique symmetric (with respect to 0)   non-increasing \color{black} function $u^\star:\mathbb T\to \R$ that has the same distribution function as $u$. $u^\star$ is called the periodic rearrangement of $u$.
 We recall that the distribution function of $u$ is
$$\mu_{ u\color{black}}(t):=\operatorname{Vol}\left(\{ u\color{black}\geq t\}\right)$$ and that $u^\star$, the periodic rearrangement of $u$, is the left inverse of $\mu_u$. 


\begin{proposition}\label{Pr:Convex}
Let $v$ the solution of (\ref{Eq:Periodic}) associated with the initial datum $u_{0}^{\star}$. Then 
$$ \forall t \in (0;T)\,, \forall r\in (0;\pi)\,, \int_{-r}^r v(t,x)dx\geq  \int_{-r}^r u^{\star}(t,x)dx.\color{black}$$
In particular, taking $t=T$ and $r=\pi$:
\begin{equation} \int_{-\pi}^\pi v(T,x)dx\geq  \int_{-\pi}^\pi u^\star(T,x)dx= \int_{-\pi}^\pi u(T,x)dx.\color{black}\end{equation}  

\end{proposition}

As   explained in the introduction, Proposition \ref{Pr:Convex} follows simply by adapting minor points in the proofs of \cite{Bandle}, and so we will simply indicate the main steps. The core idea is the following: the comparison results and Talenti-type inequalities one finds in the rearrangement literature rely on integrating the solution of the equation we are working with on its level sets and using the isoperimetric inequality. Thus, these proofs generally work only in the case of Dirichlet boundary conditions (for a recent, analogous result in the case of Robin boundary conditions we refer to \cite{AlvinooNitschTrombetti}), as these conditions guarantee that, if the solution is non-negative, none of its level sets intersects the boundary of the domain. However, in our case, since the Neumann boundary conditions allow to symmetrise the solution $u$ and to obtain a solution on the torus $\mathbb T$, the boundary of the domain is empty and the isoperimetric inequality holds, so that the proofs are identical.

\begin{proof}[Proof of Proposition \ref{Pr:Convex}] For the sake of readability, we break down the main steps in establishing the inequality
$$ \forall t \in (0;T)\,, \forall r\in (0;\pi)\,, \int_{-r}^r v(t,x)dx\geq  \int_{-r}^r u^{\star}(t,x)dx.$$

\begin{itemize}\item \underline{Comparison result for elliptic equations:} the first step is to compare the solutions of two elliptic problems. Let $\e>0$, let $\p\in L^2(\mathbb T)\,, \p\geq 0$, let $\psi\in L^2(\mathbb T)\,, \psi\geq 0$ satisfying 
$$\forall r \in (0;\pi)\,, \int_{-r}^r \psi^\star\geq \int_{-r}^r \p^\star,$$ and let $w_\p\,, z_\psi$ be the solutions to 
\begin{equation}\label{eq:w}\begin{cases}-\Delta w_\p+\e w_\p=\p&\text{ in }\mathbb T\,, 
\\ w_\p \in W^{1,2}(\mathbb T),\end{cases}\end{equation} and 
\begin{equation}\label{eqz}\begin{cases}-\Delta z_\psi+\e z_\psi=\psi^\star&\text{ in }\mathbb T\,, 
\\ z_\psi \in W^{1,2}(\mathbb T),\end{cases}\end{equation}
Then there holds:
\begin{equation}\label{eq:Tale}
\forall r\in (0;\pi)\,, \int_{-r}^r z_\psi\geq \int_{-r}^r w_\p.\end{equation}
To obtain \eqref{eq:Tale}, we may follow the  standard steps of \cite{Talenti}: we assume that the level sets of \eqref{eq:w} have measure zero (to cover the case of level sets of positive measure, one can argue as in \cite{Talenti}, to which we refer for the sake of brevity). Let $\tau >0$ be a real number. Integrating \eqref{eq:w} on $\{w_\p\geq \tau\}$ yields
$$-\int_{\{w_\p=\tau\}}\frac{\partial w_\p}{\partial \nu}=\int_{\{w_\p\geq \tau\}}\p -\e\int_{\{w_\p\geq \tau\}}w_\p\leq \int_0^{\mu_{w_\p}(\tau)}\p^\star -\e \int_0^{\mu_{w_\p}(\tau)}w_\p^\star,$$ where the last inequality comes from the Hardy-Littlewood inequality.  We recall that from the co-area formula
\begin{equation}\text{ for a.e. $\tau$, }\mu_{w_\p}'(\tau)=-\int_{\{w_\p=\tau\}}\frac1{|\n w_\p|}.\end{equation} From the Cauchy-Schwarz inequality and the isoperimetric inequality we obtain
\begin{align*}
4&\leq \operatorname{Per}\left(\{w_\p=\tau\}\right)
\\&\leq \int_{\{w_\p=\tau\}}\frac1{|\n w_\p|}\int_{\{w_\p=\tau\}}|\n w_\p|
\\&\leq -\mu_{w_\p}'(\tau)\int_{\{w_\p=\tau\}}|\n w_\p|\leq -\mu_{w_\p}'(\tau)\left(\int_0^{\mu_{w_\p}(\tau)}\p^\star -\e \int_0^{\mu_{w_\p}(\tau)}w_\p^\star\right).
\end{align*} Since $w_\p^\star$ is the left inverse of $\mu_{w_\p}$,standard arguments \cite{Alvino1990} imply
\begin{equation}\forall \xi\in (0;\pi)\,,-4(w_\p^\star)'(\xi)\leq \int_0^\xi \p^\star-\e\ \int_0^\xi w_\p^\star.\end{equation}
It should be noted that if we work with \eqref{eqz} instead of \eqref{eq:w} every inequality becomes an equality since $z_\psi=z_\psi^\star$, and thus $z_\psi$ satisfies
\begin{equation}\label{eeq:tz}
-4(z_\psi^\star)'(\xi)= \int_0^\xi \psi^\star-\e\ \int_0^\xi z_\psi^\star.
\end{equation}
Defining $Z_\p:=\int_0^\xi (z_\p-w_\p^\star)$ we hence have 
$$-Z_\p''+\frac\e4Z_\p\geq 0\,, Z_\p(0)=0.$$ Furthermore, integrating \eqref{eq:w} and \eqref{eqz} on the torus we obtain, by the equimeasurability of a function and its rearrangement, 
\begin{equation}
\int_0^{\pi} w_\p^\star=\frac12\int_{\mathbb T}w_\p=\frac1{2\e}\int_{\mathbb T}\p\leq \frac1{ 2\e}\int_0^\pi \psi=\int_0^\pi z_\psi^\star\end{equation} so that 
$$Z_\p(\pi)\geq 0.$$ By the maximum principle, $Z_\p\geq 0$, which concludes the proof.

\item \underline{Comparison result for parabolic equations:} for this second step, we follow the strategy of \cite{Bandle}, which relies on a Picard iteration scheme. Namely, we discretise the parabolic problem in time: let $N\in \N^*$ be a discretisation step. For $u_0\in \mathcal A$, we define the sequences $\{u_{0,k}\,, v_k\}_{k=0\,, \dots\,, N}$ as the solutions to 
\begin{equation}\label{eq:uk}u_{0,0}=u_0\,, \forall k\in \{0\,, \dots, N_1\}\,,\begin{cases}-\Delta u_{0,k+1}+\frac1Nu_{0,k+1}=\frac1Nu_{0,k}+g(u_{0,k})&\text{ in }\mathbb T\,, 
\\ u_{0,k}\in W^{1,2}(\mathbb T),\end{cases}\end{equation} and 
\begin{equation}\label{eq:vk}v_{0,0}=u_0^\star\,, \forall k\in \{0\,, \dots, N_1\}\,,\begin{cases}-\Delta v_{0,k+1}+\frac1Nv_{0,k+1}=\frac1Nv_{0,k}+g(v_{0,k})&\text{ in }\mathbb T\,, 
\\ v_{0,k}\in W^{1,2}(\mathbb T)\end{cases}\end{equation} respectively. As $g$ is convex and increasing, we have $g(v)^\star=g(v^\star)$. Thus, we can prove inductively that for every $N$ and for every $k\in \{0\,, \dots\,, N\}$ there holds 
$$\forall r\in (0;\pi)\,, \int_{-r}^r v_{0,k}\geq \int_{-r}^r u_{0,k}^\star$$ and it remains to pass to the limit $N\to \infty$ to recover the result.
 \color{black}
\end{itemize}
\end{proof}

\paragraph{Conclusion.}
Assume that $u_{0}$ is a bang-bang maximiser of problem $\eqref{Eq:Pv}$. Symmetrise it and extend it by periodicity. Consider the symmetric decreasing rearrangement $u_{0}^{\star}$ of its extension. Then 
$\int_{-\pi}^\pi v(T,x)dx\geq \int_{-\pi}^\pi u(T,x)dx$ by Proposition \ref{Pr:Convex}, where $v$ is the solution of the periodic equation (\ref{Eq:Periodic}) associated with the initial datum $u_{0}^{\star}$. Clearly, $v(t,\cdot)$ and $u(t,\cdot)$ are symmetric with respect to $x=0$ for all time $t>0$. Hence,  $\int_{0}^\pi v(T,x)dx\geq \int_{0}^\pi u(T,x)dx$. Also, one easily remarks that $v$ restricted to $(0,\pi)$ is the solution of the parabolic equation with Neumann boundary conditions (\ref{eq:rd_Gral}), associated with the initial datum $u_{0}^{\star}$ restricted to $(0,\pi)$. On the other hand, as $u_{0}$ is bang-bang, one has $u_{0}^{\star}=\mathds{1}_{(-m,m)}$. 
Hence, $\mathds{1}_{(0,m)}$ increases the criterion in $\eqref{Eq:Pv}$. Thus, it is a solution of $\eqref{Eq:Pv}$. \color{black}

\section{Numerical analysis in the bistable framework} \label{sec:numerics}
As we explained in the introduction, the behaviour of optimisers vary wildly depending on the shape of the reaction term $f$. To exemplify this phenomenon, we use the bistable non-linearity that motivated \cite{NadinToledo}, namely, $f(u):=u(1-u)(u-\theta)$, with $\theta \in (0,1)$. 

When considering the optimisation problem \eqref{Eq:Pv}, the fact that the set $\{p=c\}$ may have a positive measure or, in other words, that a solution may not be the characteristic function of a set, leads to several difficulties in terms of numerical methods, because standard gradient methods or fixed-point algorithms fail to capture what this so-called "singular arc" should be replaced with. 

Let us first recall the main principles of the numerical algorithm introduced in \cite{NadinToledo} and explain the difficulty related to  $\{p=c\}$ further. 
Given the initial condition at the $n$-step $u_0^n$, we construct $u_0^{n+1}=u_0^n+h_0^n$, where $h_0^n$ maximises \eqref{Derivative:J_T} and is an admissible perturbation. Since the adjoint at the $n$-th step $p_0^n$ may have level sets of positive measure, one can not directly apply the bathtub principle and choose $h_0^n$ as the difference of characteristic functions of two level sets of $p_0^n$; we must thus describe what happens on the singular arc, that is, on the level set $\{p_0^n=c^n\}$ where $c^n$ is chosen so that 
\begin{equation}|\{p_0^n>c^n\}|<m\,, |\{p_0^n\geq c^n\}|>m\,, |\{p_0^n=c^n\}|>0.\end{equation} We first define, in this case, $u_0^{n+1}=1$ on $\{p_0^n>c^n\}$, $u_0^{n+1}=0$ on $\{p_0^n<c^n\}$, and it remains to fix the value of $u_0^{n+1}$ on $\omega_n$. Defining $\omega_n:=\{p_0^n=c^n\}$ and discretising equation \eqref{eq:rd_p} on $\omega_n$ we obtain, with an explicit finite difference scheme
\begin{equation}\label{Eq:Hazard}
-\left(\frac{p_0^n(dt,x)-c^n}{dt}\right)=f'(u_0^{n+1})c^n\end{equation} and the value on $u_0^{n+1}$ on $\omega_n$ must be a root of \eqref{Eq:Hazard}. However, for bistable non-linearities, this equation may have two roots, say $\mu_1^n$ and $\mu_2^n$. In this case, these two roots can be distinguished through the convexity of $f$. In other words, if we have two roots, up to relabelling, 
\begin{equation}f''(\mu_1^n)>0\,, f''(\mu_2^n)<0.\end{equation} 
In \cite{NadinToledo} this difficulty is overcome by examining the two different possibilities and choosing the best one, which significantly lessens the performance of the algorithm, but Theorem \ref{Prop1} allows to overcome this difficulty by choosing directly the root $\mu_2^n$, which is in the "concavity" zone of $f$.

\subsection{Comparison of different numerical methods in the one-dimensional case}

In this section, we want to study an example in order to compare the performances of our numerical algorithm with other well known optimisation algorithms to solve general nonlinear problems under constraints. More precisely, we will consider the following numerical methods: 
\begin{itemize}
	\item \emph{Method 1:} The numerical algorithm introduced in \cite{NadinToledo}, which we improve using Theorem \ref{Prop1}, and that will be referred to as \emph{our algorithm}.
	\item \emph{Method 2:} The \textit{interior-point} method, which is used to solve optimisation problems with linear equality and inequality constraints by applying the Newton method to a sequence of equality constrained problems. For a more detailed description of this method see for instance \cite{Boyd_04}.
	\item \emph{Method 3:} The \textit{sequential quadratic programming} (SQP), which solves a sequence of optimisation sub-problems, each of which optimizes a quadratic model of the objective function subject to a linearisation of the constraints, see for instance \cite{Nocedal_06}. 
	\item \emph{Method 4:} The \textit{simulated annealing} method, which is a probabilistic technique used to approximate global optimisation in a large search space. See for instance \cite{Locatelli_00} for more details on this technique.
\end{itemize}

We used the MATLAB platform to perform the simulations. Methods 2 and 3 are already coded in the MATLAB function \verb+"fmincon"+ while methods 1 and 4 were coded for the experiment.  

\paragraph{Setting the data}

Let us consider $\Omega=(-50;50)$, and $m=13$; thus the admissible set is defined as follows:
\[
\mathcal{A}_{13}=\left \{u_0\in L^1(\Omega): 0 \leq u_0(x)\leq 1, \hbox{ and } \int_{\Omega}u_0(x) dx = 13 \right \}.
\]
Note that this set is defined by two inequalities and an equality constraint. We aim at maximising the quantity $\mathcal{J}_{T}(u_0):=\int_{\Omega}u(T,x) dx$ for $T=25$ and we use a bistable reaction term $f(u):= u(1-u)(u-0.25)$. 

In order to compare the performance of the four algorithms under the same conditions, we consider the same discretisation of $\Omega$. Moreover, the solution of the equation is systematically computed by the Crank-Nicolson method, and, for the initialisation we consider the same $u_0^0$ given by a single block of mass $13$. 
The value of the objective function at each iteration is numerically approximated by the rectangle rule. In particular, for the initialisation we have $\mathcal{J}_{25}(u_0^0)=29.42$. 


The results of the simulations are shown in Fig. \ref{Comparison} and Table \ref{Tab:Numerics}. For this example, our algorithm turns out to be faster than other well-known algorithms. Moreover, the evaluation of the objective function differs in less than $1 \%$ with respect to the best result obtained with the sequential quadratic programming method which takes more than twice the run-time of our algorithm.

\begin{table}[h]
	\caption{Comparing algorithms}
	\label{Tab:Numerics}  
	\centering    
	\begin{tabular}{lcc}
		\hline\noalign{\smallskip}
		Algorithm & {\begin{tabular}{c} Objective function \\ $\mathcal{J}_{25}(\overline{u}_0)$ \end{tabular} } & {\begin{tabular}{c} Run-time\\ (in seconds) \end{tabular} } \\
		\noalign{\smallskip}\hline\noalign{\smallskip}
		Our algorithm  & $77.9864$ & $452$ \\
		Interior point & $65.6175$ & $676$ \\
		Sequential quadratic programming  & $78.7672$ & $1342$ \\
		Simulated annealing & $77.6238$ & $4148$ \\
		\noalign{\smallskip}\hline
	\end{tabular}
\end{table}

\begin{figure}
	\centering
	\subfloat[Our algorithm]{\includegraphics[width = 3.5in]{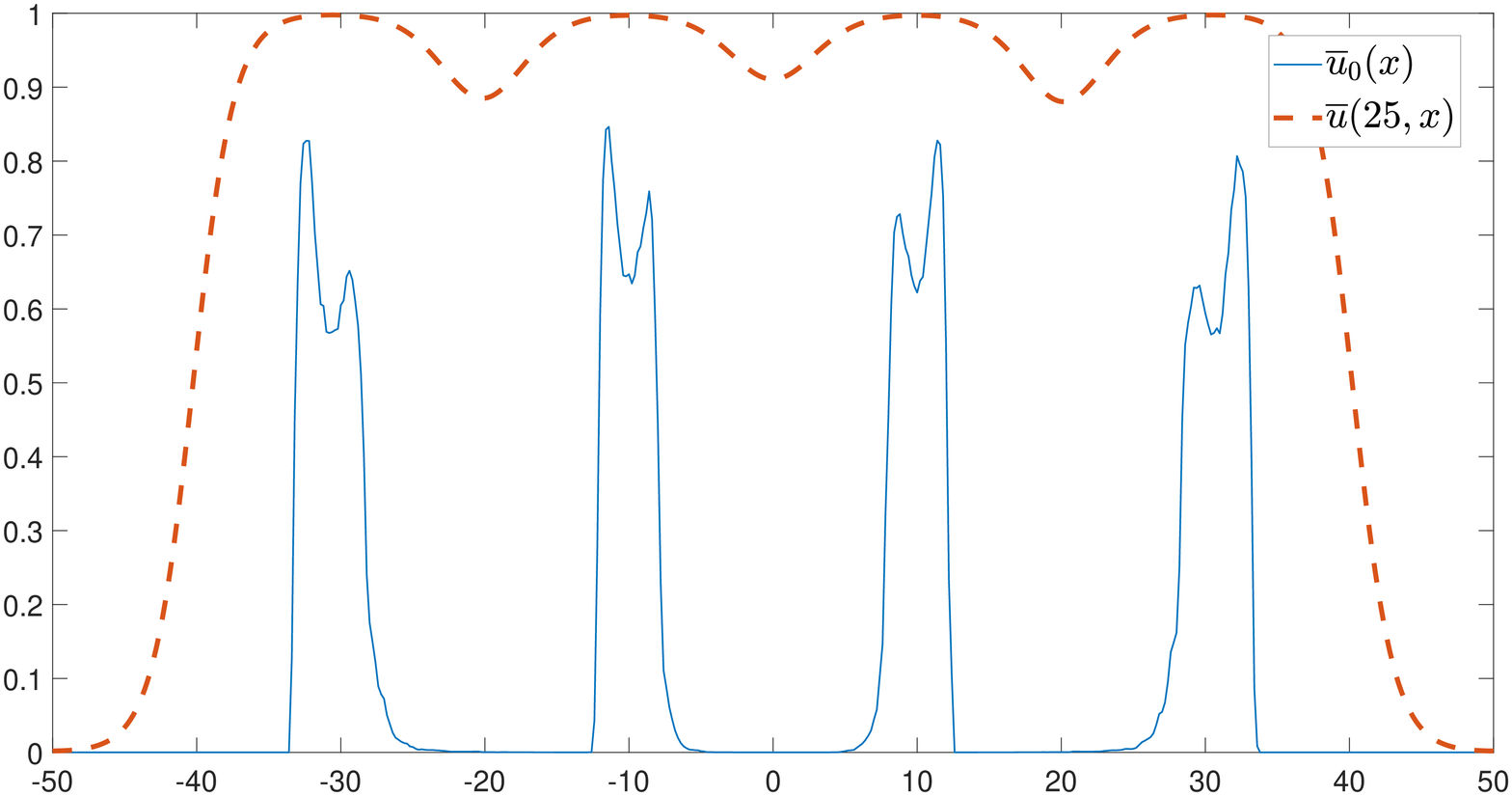}} \\
	\subfloat[Interior-point method]{\includegraphics[width = 3.5in]{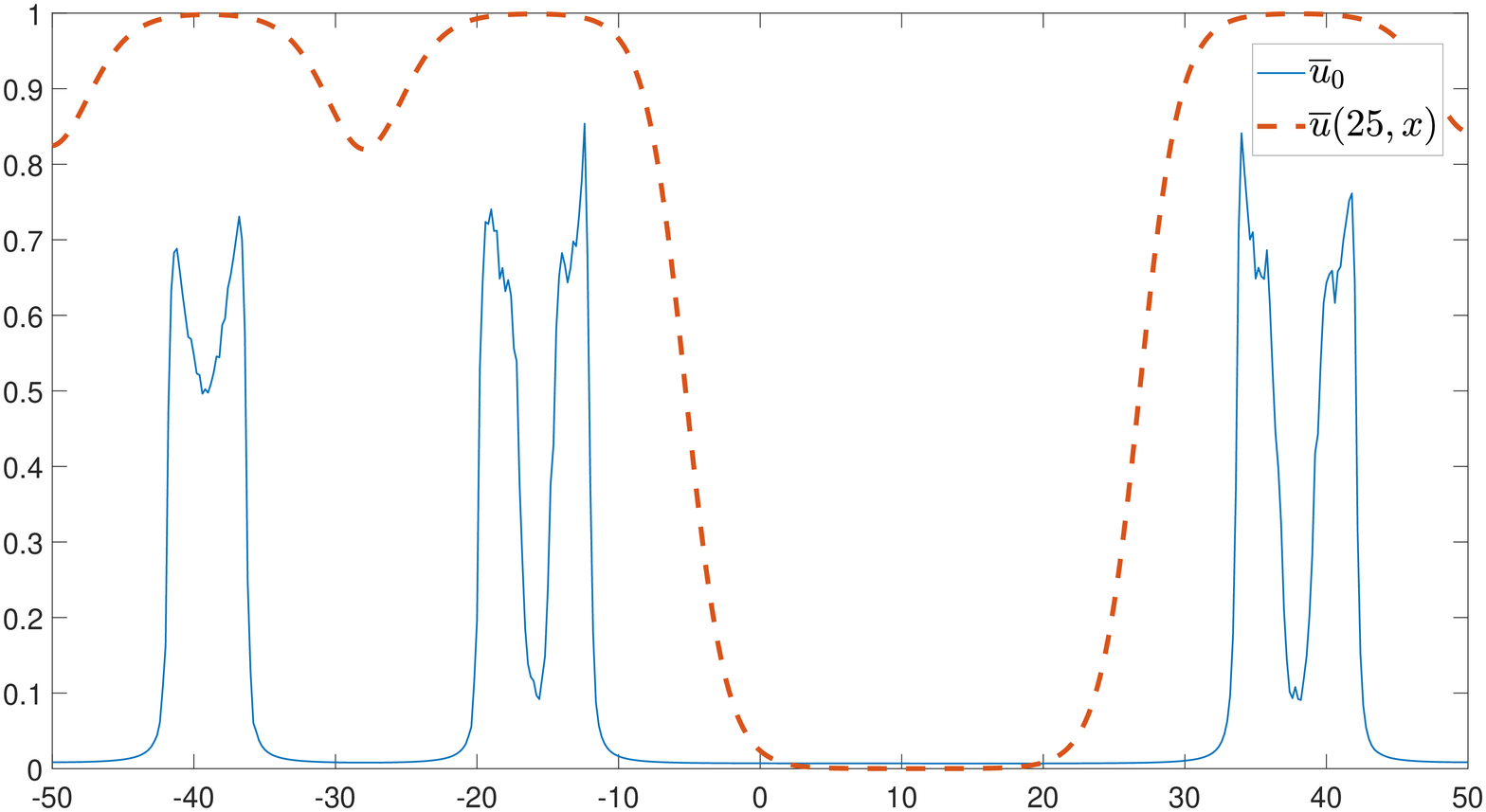}}\\
	\subfloat[Sequential quadratic programming]{\includegraphics[width = 3.5in]{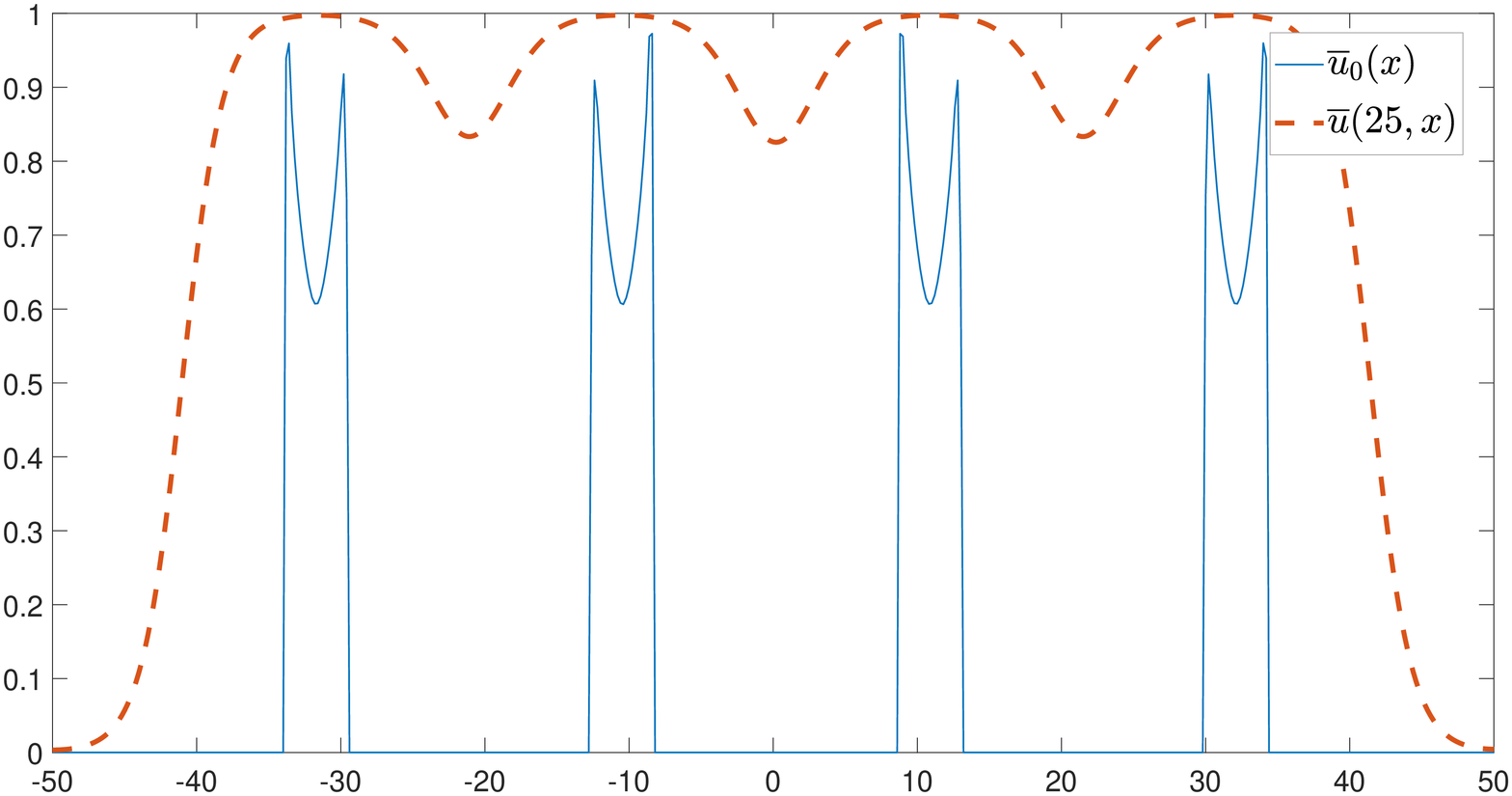}}\\
	\subfloat[Simulated annealing]{\includegraphics[width = 3.5in]{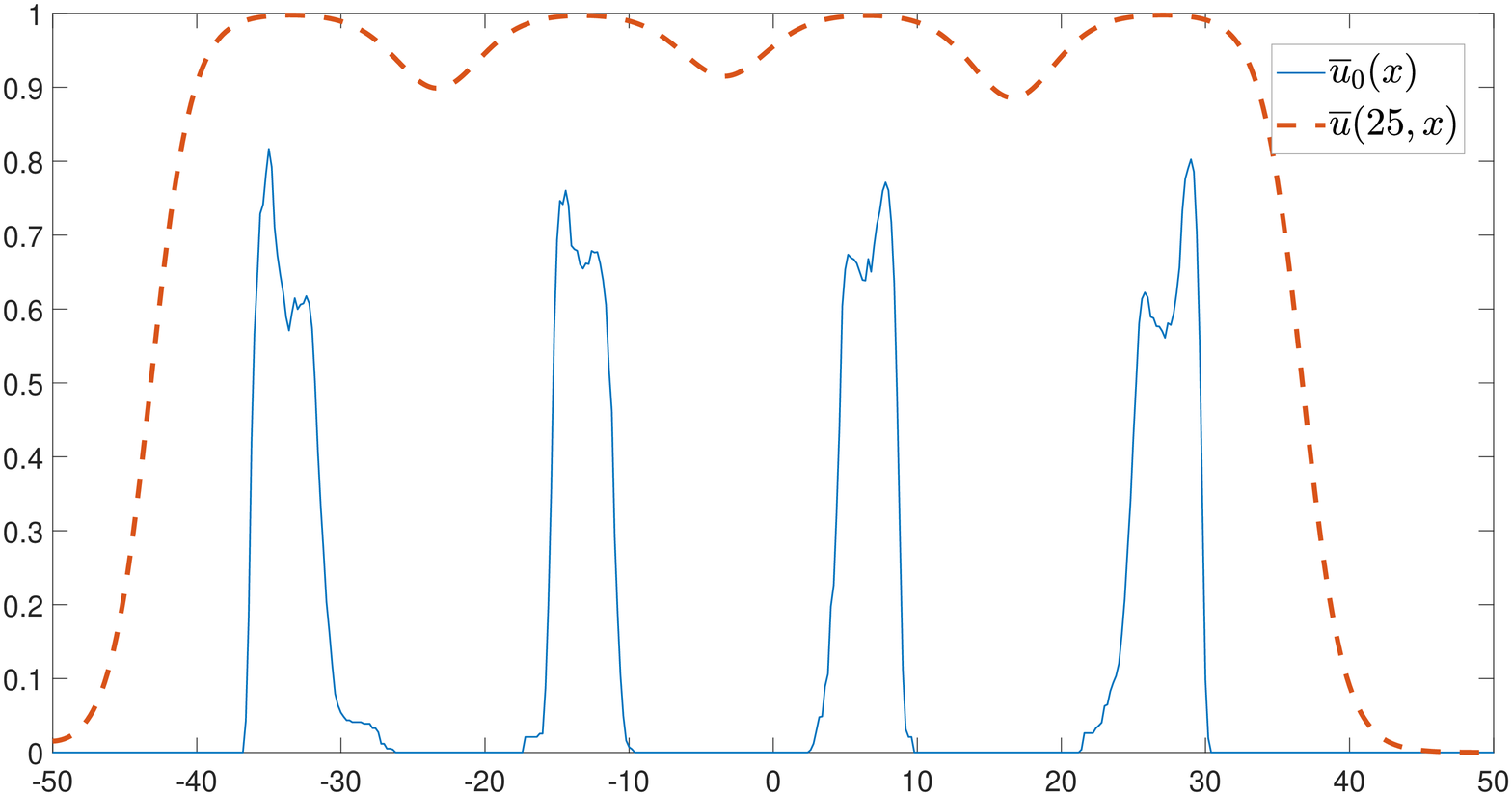}}
	\caption[Comparison of the numerical algorithms]{Optimum found by means of the four different numerical algorithms.}
	\label{Comparison}
\end{figure}

Though the solution given by the sequential quadratic programming method is clearly more regular than the others, the profile of the local optimisers found by simulated annealing and by our algorithm do not seem to be far from this profile.   Indeed, the solutions obtained through Methods 1, 3 and 4 are qualitatively similar. On the other hand, the interior-point method gives a significantly different optimum, which seems to point out the good performance of our algorithm. \color{black} It is important, however, to highlight that since uniqueness is not guaranteed in general, one can not ensure that the algorithms have converged to a global maximiser but only to a local one.

\subsection{Numerical simulations in the two-dimensional case}

We only considered in the present paper the one-dimensional case. We now display some numerical results obtained in dimension $2$, for which new patterns might arise.

To solve the reaction-diffusion equation in the two-dimensional case, we consider the alternating direction implicit method (ADI) which is a classical method to solve parabolic problems in two or three dimensions. As in the one-dimensional case, the algorithm and routines were coded in MATLAB.

We consider a square domain $\O=(-10;10)\times(-10;10)$, discretised uniformly by squares of side    $dx=0.22$ \color{black}. We fix $T=30$ for all subsequent simulations.   We first tackle the case of a bistable reaction term $$f(u)=u(1-u)(u-0.25).$$ In a second paragraph, we study the monostable case $$f(u)=(u+0.25)u(1-u).$$ The justification for this second case is that this is a non-concave monostable non-linearity. It is hence not covered by the theoretical results of \cite{NadinToledo}.
\subsubsection{The bistable case} 
\color{black}
\paragraph{Example 1}

The algorithm is initialised with a ball of full density located in the middle of the domain $\O$. The mass is fixed to $m=5.8\pi$, see Fig. \ref{Ejemplo 1}(a). After 20 iterations, the algorithm converges to the local optimum showed in Fig. \ref{Ejemplo 1}(b). The evolution of the objective function $\mathcal{J}_{30}$ through iterations is showed in Fig. \ref{Fig:operador_2}

\begin{figure}[!h]
	\centering
	\subfloat[$u_0^0$]{\includegraphics[width = 0.5\textwidth]{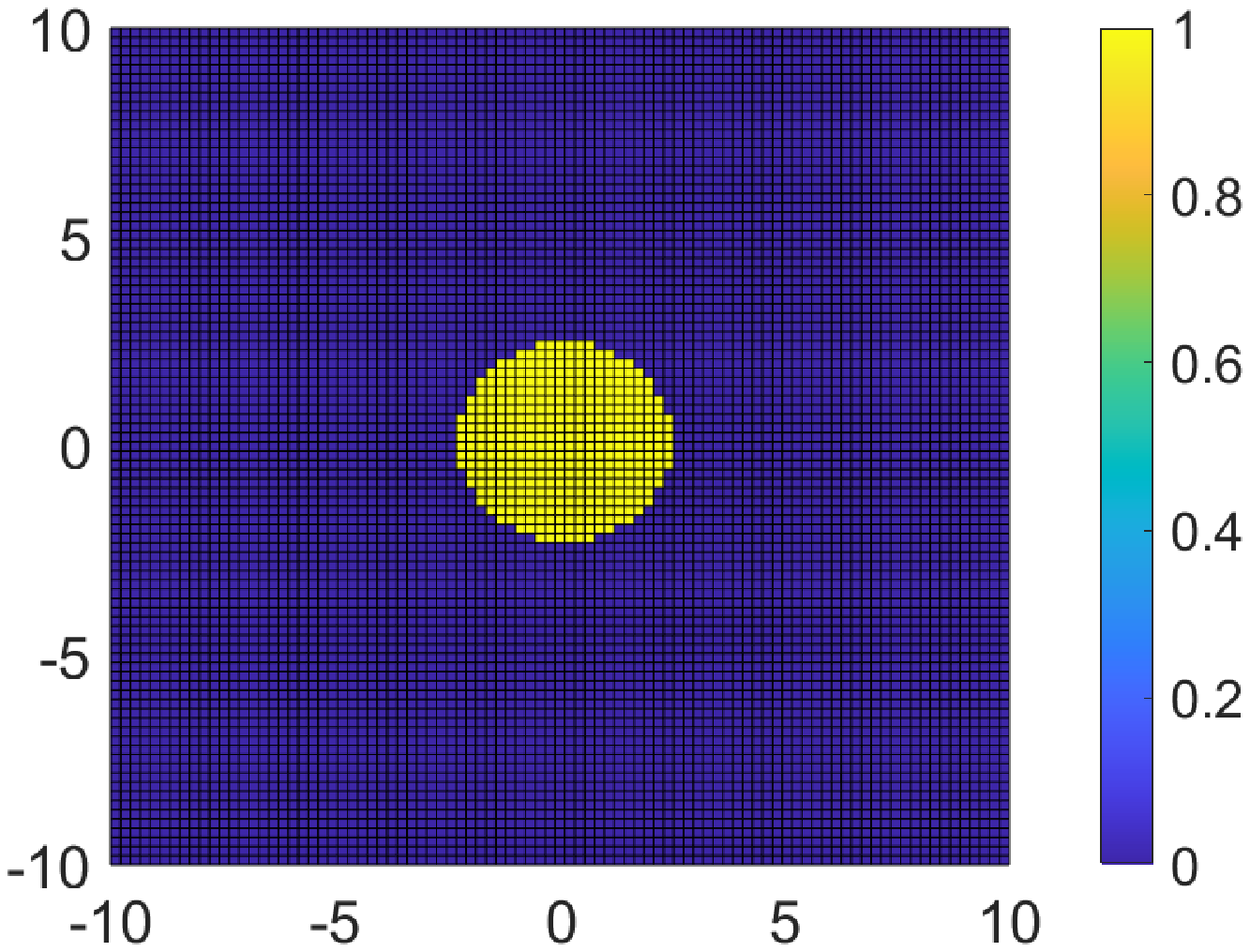}}
	\subfloat[$\overline{u}_0=u_0^{20}$]{\includegraphics[width = 0.5\textwidth]{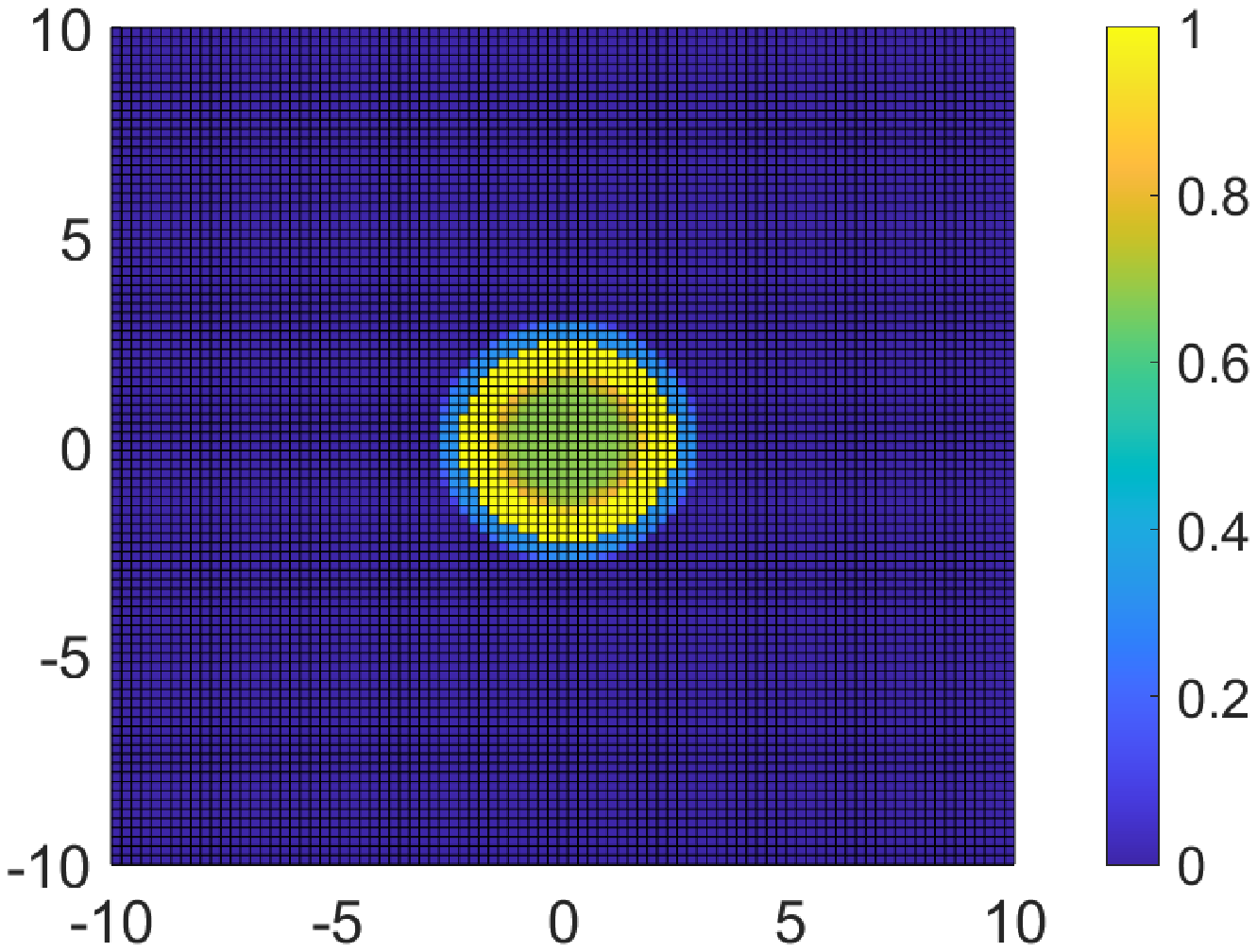}}
	\caption[2D-simulations example 1, local optimum]{In the left-hand side, we show the input of the algorithm, given by the ball of radius $r=\sqrt{5.8}$ centered at the origin. In the right-hand side,  we display the local optimum found by the numerical algorithm after 20 iterations, which looks radial, but is no longer a bang-bang distribution: it does not only take values 1 and 0.}
	\label{Ejemplo 1}
\end{figure}

\begin{figure}[!h]
	\centering
	\includegraphics[width=0.7\textwidth]{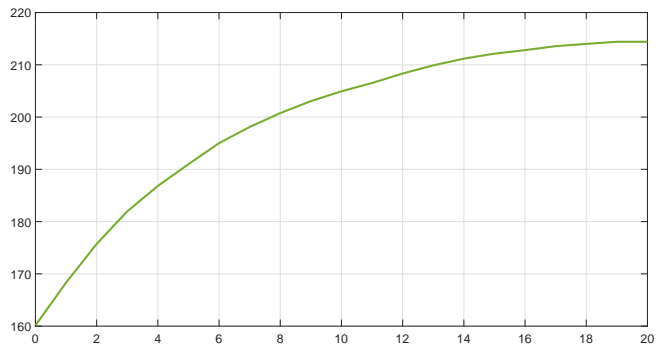}
	\caption[Evolution of the objective function example 1]{Evolution of the objective function from the initialisation $\mathcal{J}_{30}(u_0^0)=160.1$ to the last iteration $\mathcal{J}_{30}(u_0^{20})=214.4$.}
	\label{Fig:operador_2}
\end{figure}

One might see that the local optimum found by the numerical algorithm is no longer a bang-bang function but a circular ball with less mass in the middle and a slightly bigger ratio. Looking at the adjoint state defined as the solution of equation \eqref{eq:rd_p}, associated to this initial data, one might see that the area in the middle of the circle corresponds to a set where the adjoint state remains constant, see Fig. \ref{Ejemplo 2 adjunto}.

\begin{figure}[!h]
	\centering
	\includegraphics[width = 0.7\textwidth]{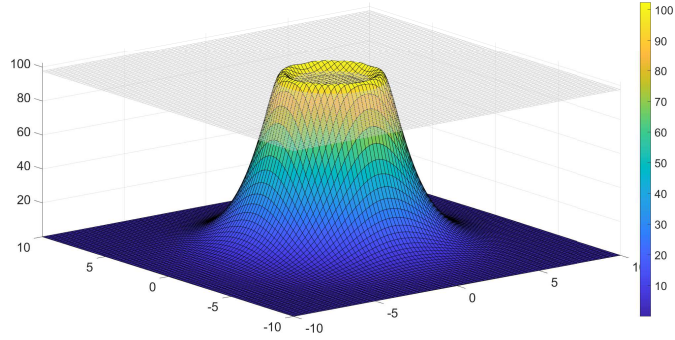}
	\caption[2D simulations example 1, adjoint state]{The figure shows the surface given by the solution $\overline{p}(0,x)$ of the adjoint problem defined by the eq. \ref{eq:rd_p} associated to the initial data $\overline{u}_{0}$ found by our algorithm. The plane colored in gray, is associated to the value $\overline{c}$ described in Theorem \ref{Thm1} and thus for every $x\in \O$ such that $p_0(\mathbf{x})=\overline{c}$, one has $0<\overline{u}_0(x)<1$, see Fig. \ref{Ejemplo 1} (b).}
	\label{Ejemplo 2 adjunto}
\end{figure}

\paragraph{Example 2}

In this case, we keep the same discretization and initial mass $m$ of the previous example, but we consider an initial data which is a stripe of full density dividing our domain into two equal regions of zero density, see Fig. \ref{Ejemplo 2} (a). The algorithm converges after 38 iterations and the local optimum is displayed in Fig. \ref{Ejemplo 2} (b). The corresponding variations of the objective function is showed in Fig. \ref{Fig:operador_3}

\begin{figure}[!h]
	\centering
	\subfloat[$u_0^0$]{\includegraphics[width = 0.5\textwidth]{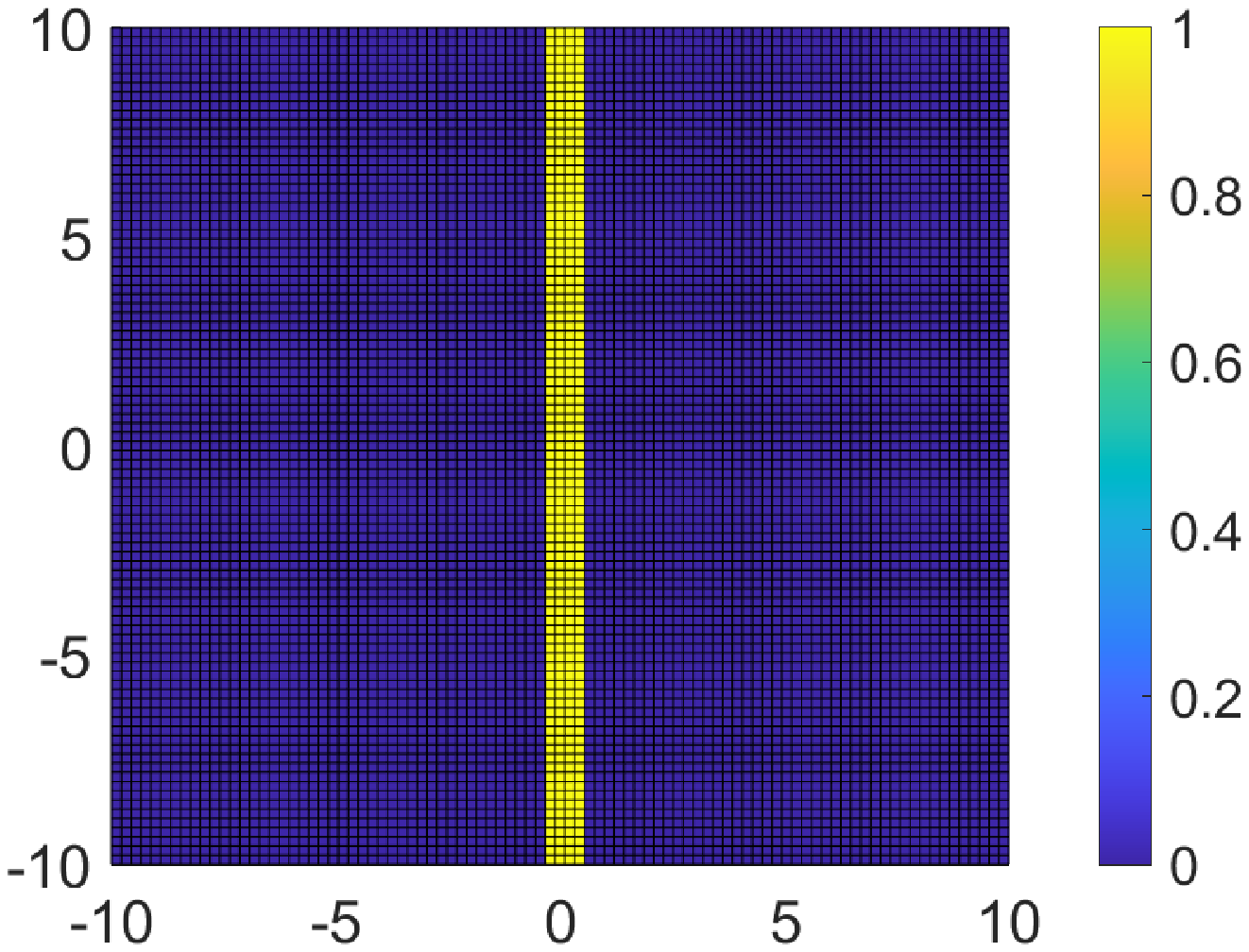}}
	\subfloat[$\overline{u}_0=u_0^{38}$]{\includegraphics[width = 0.5\textwidth]{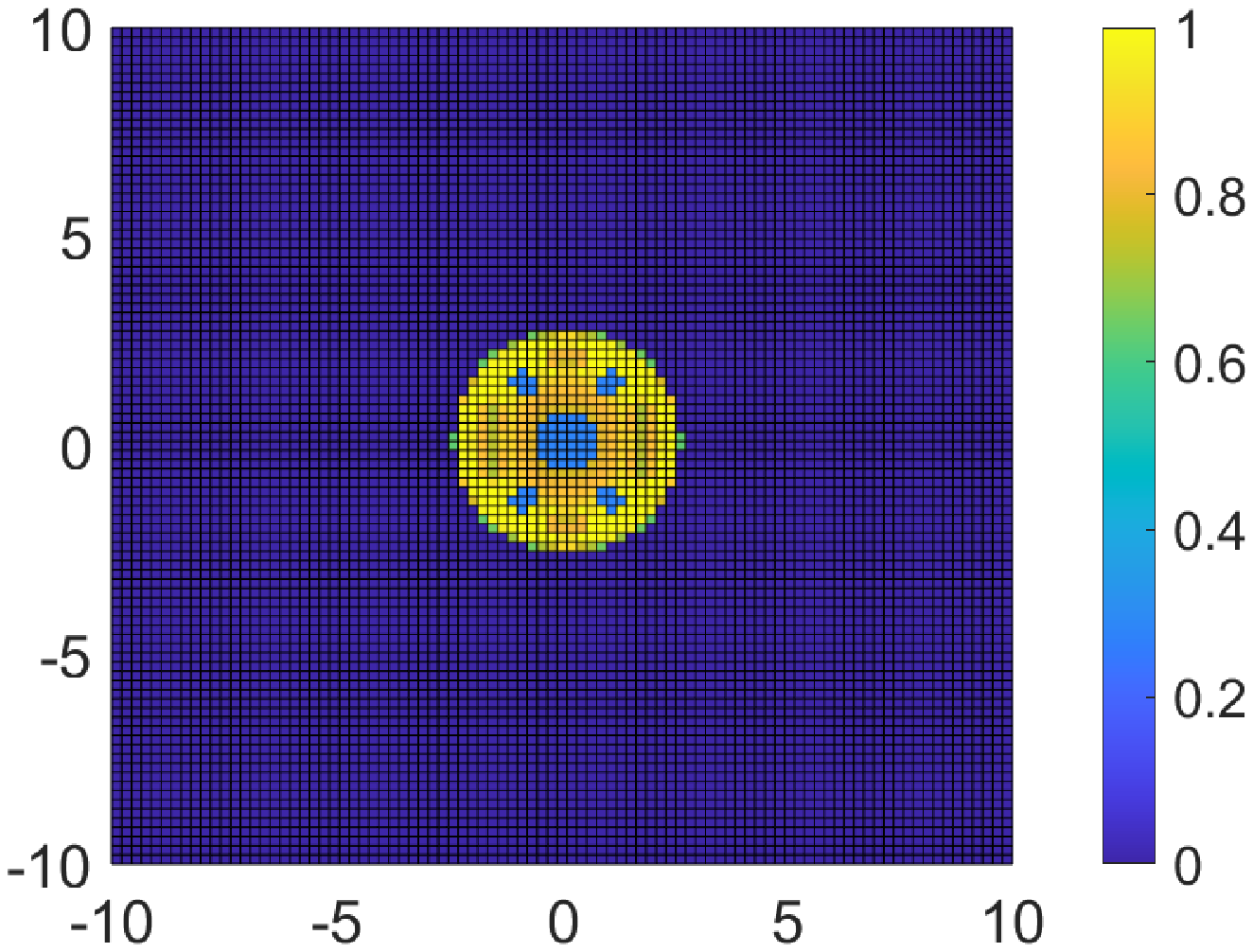}}
	\caption[2D-simulations example 2, local optimum]{In the left-hand side is showed the input of the algorithm, given by the stripe of width $r=0.91$ centred at the origin. In the right-hand side, the local optimum found by the numerical algorithm after 38 iterations.}
	\label{Ejemplo 2}
\end{figure}
\bigskip
\begin{figure}[!h]
	\centering
	\includegraphics[width=0.7\textwidth]{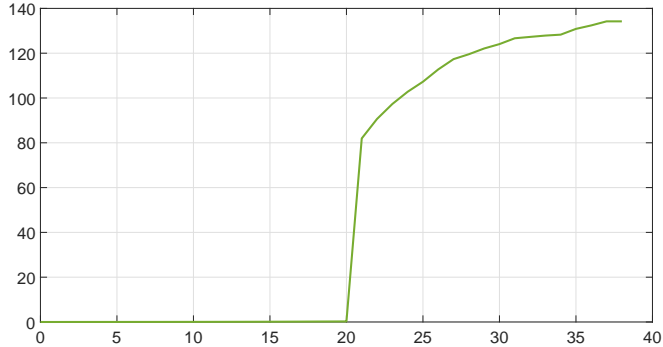}
	\caption[Evolution of the objective function example 2]{Evolution of the objective function from the initialisation $\mathcal{J}_{30}(u_0^0)=5.5\times 10^{-6}$ to the last iteration $\mathcal{J}_{30}(u_0^{38})=134.2$.} 
	\label{Fig:operador_3}
\end{figure}

We observe that, in this case, the value of the objective function remains very low during the first 20 iterations. This fact, together with the radial geometry of the optimum found by the algorithm suggests that this stripe geometry is not optimal. It should also be pointed out that the geometry of the local optimum is interesting: indeed, it shows regions of zero density (i.e. the optimum $\overline u_0$ found by the algorithm is equal to 0 in theses regions) in the middle of regions of full density (i.e. where $\overline u_0=1$), which exemplifies the phenomenon described in the one-dimensional case in \cite{GHR}.

Another relevant feature is that the optima found in the first and second examples are different, which indicates that our algorithm converge to local optima, and thus that the choice of the initial distribution $u_0^0$ is crucial.\color{black}

\paragraph{Example 3}

For this example we keep the settings of the previous one, but we consider a higher initial mass $m=27$. The geometry of the initial distribution is a stripe of full density dividing the domain into two regions of zero density, like in the example 2, see Fig. \ref{Ejemplo 3}(a). 

The corresponding local optimum found by the numerical algorithm is showed in Fig. \ref{Ejemplo 3}(b). As in the previous case, the optimiser reflects a low density zone ringed by a high density region. This gap is clearly filled by diffusion as time evolves. 

\begin{figure}[!h]
	\centering
	\subfloat[$u_0^0$]{\includegraphics[width = 0.5\textwidth]{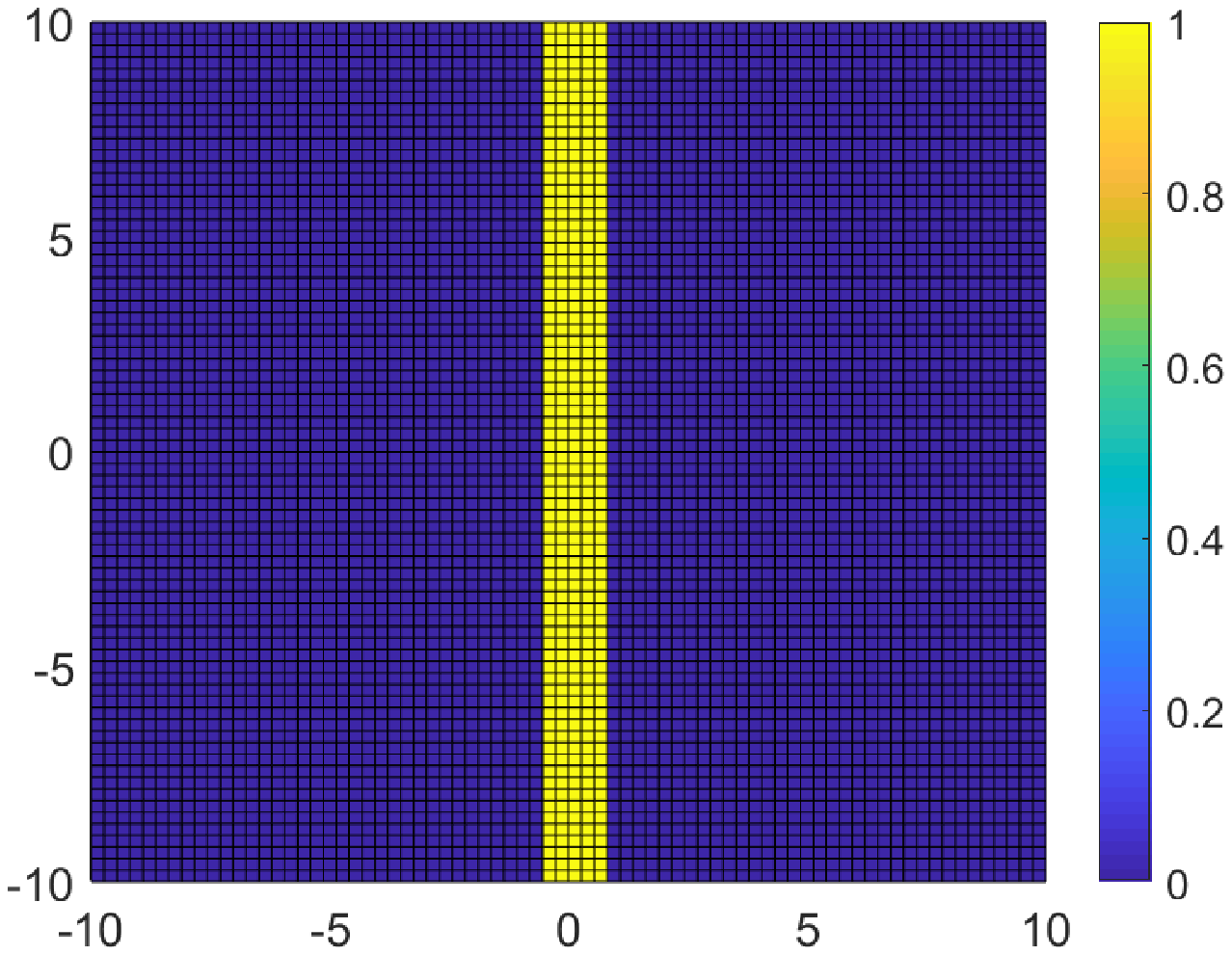}}
	\subfloat[$\overline{u}_0=u_0^{50}$]{\includegraphics[width = 0.5\textwidth]{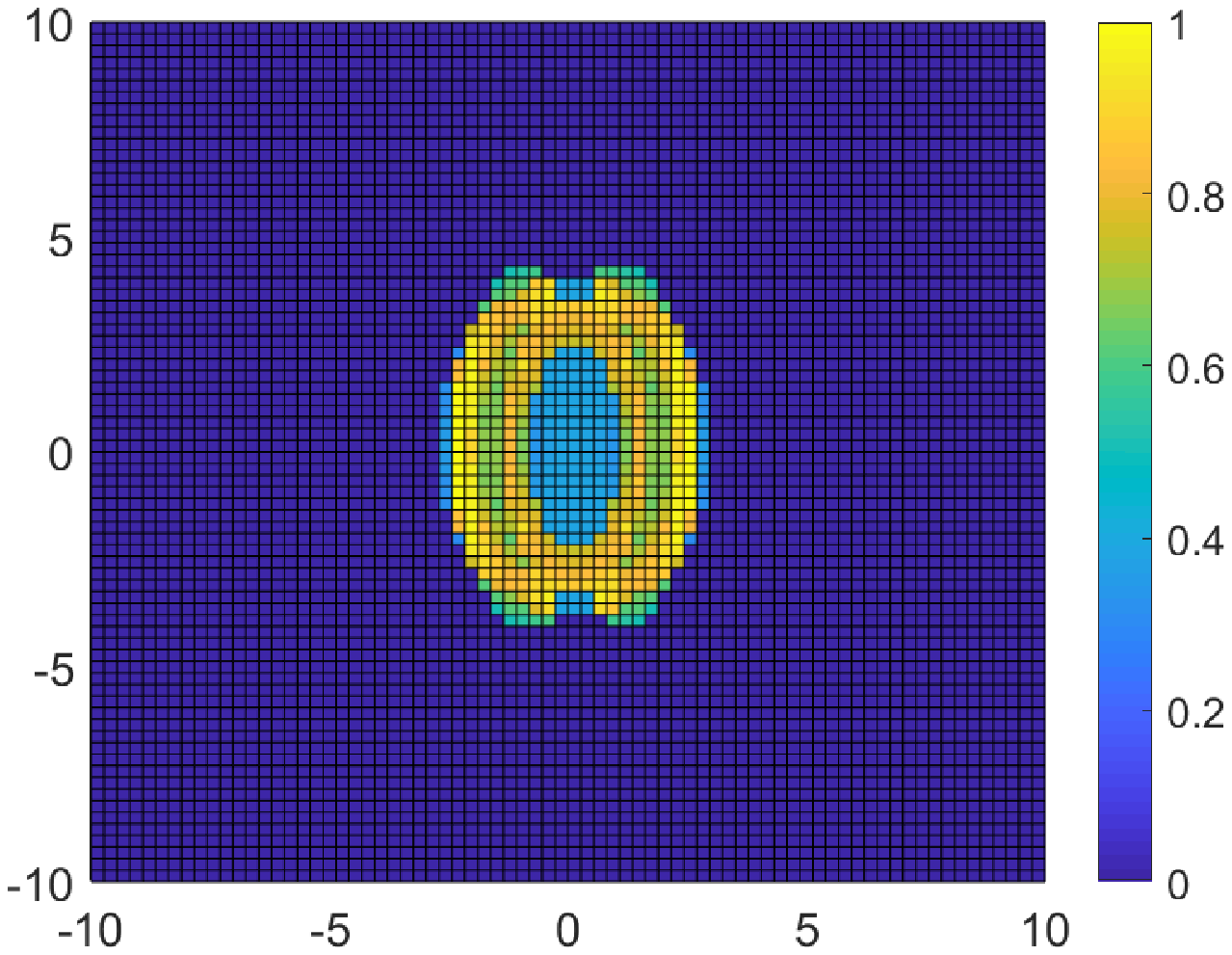}}
	\caption[2D-simulations example 3, local optimum]{In the left-hand side is showed the input of the algorithm, given by the stripe of width $r=1.6$ centered at the origin. In the right-hand side, the local optimum found by the numerical algorithm after 50 iterations.}%
	\label{Ejemplo 3}
\end{figure}

\begin{figure}[!h]
	\centering
	\includegraphics[width=0.7\textwidth]{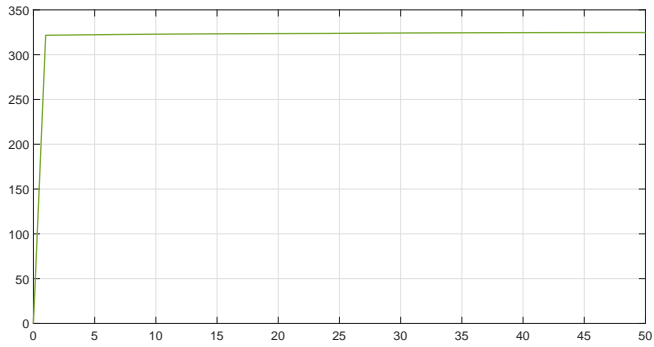}
	\caption[Evolution of the objective function example 3]{Evolution of the objective function from the initialisation $\mathcal{J}_{30}(u_0^0)=3\times 10^{-5}$ to the last iteration $\mathcal{J}_{30}(u_0^{50})=324$.}%
	\label{Fig:operador_4}
\end{figure}

This example suggests  once again the non optimality of stripe-like initial distributions. Note from Fig. \ref{Fig:operador_4} that, despite the considerable increase of the initial mass with respect to example 2, the values of the objective function $\mathcal{J}_{30}(u_0^0)$ associated to the stripe is of the order of $10^{-5}$, which is very low compared with the value associated to the final distribution. 

\bigskip

Finally, let us mention that a possible approximation of the maximiser was discussed in the Appendix of \cite{TheseAna}. Namely, in this thesis, the author replaced the maximiser $\overline{u}_{0}$ by its mean on each of the connected components of $\{\overline{u}_{0}>0\}$. This gives pretty good results in several cases. It would be good to manage to quantify analytically the difference of criterion between this approximated initial datum and the global maximiser.
%

 \subsubsection{The non-concave monostable case}\label{e:f}
 We now present some simulations in the case $$f(u)=(u+0.25)u(1-u),$$ still working under the constraint that $0\leq u_0\leq 1$. The motivation behind this case is that this non-linearity is monostable on $(0;1)$, i.e. it only has one stable equilibrium, but it is not concave. As a consequence, the theoretical approach developed in \cite{NadinToledo} can not guarantee that the optimiser $\overline u_0$ is the characteristic function of a subset of $\O$.
 
The parameters of the simulations are still the same: $T=30$, and the initial configuration is the same as in Example 1: the initialisation is a ball of full density, with mass $m=5.8\pi$ located in the middle of the domain $\O$.  We refer to Fig. \ref{Ejemplo 10} and \ref{Ejemplo11}. We should however point out that in this simulation, the value of the objective function remains almost constant, despite the fact that the final configuration is very different from the initial one. It is plausible that such monostable non-linearities converge too quickly to the equilibrium. 

\begin{figure}[!h]
	\centering
	\subfloat[$u_0^0$]{\includegraphics[width = 0.5\textwidth]{Figures//U_0_0_edit.eps}}
	\subfloat[$\overline{u}_0=u_0^{50}$]{\includegraphics[width = 0.5\textwidth]{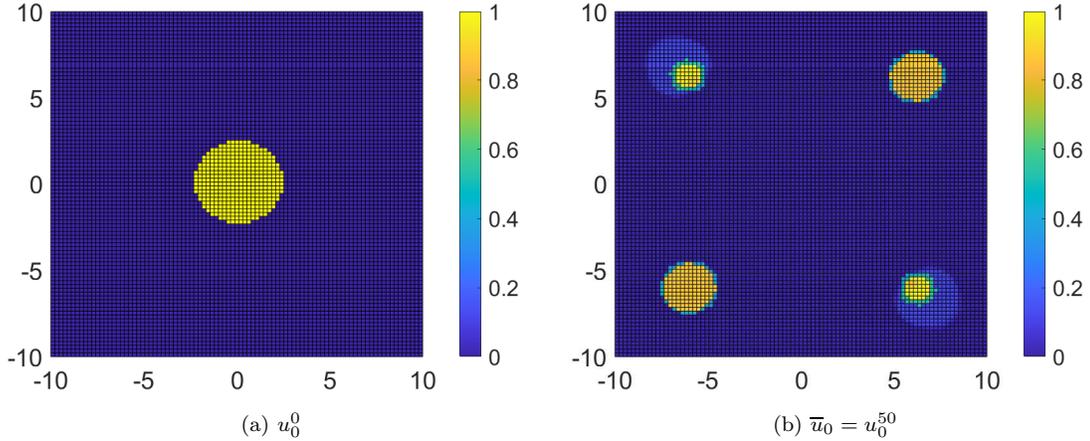}}
	\caption[2D-simulations example, local optimum]{In the left-hand side we display the input of the algorithm, given by a centered ball. In the right-hand side, we show the local optimum found by the numerical algorithm after 50 iterations.}%
	\label{Ejemplo 10}
\end{figure}

\begin{figure}[!h]
	\centering
	\includegraphics[width=0.7\textwidth]{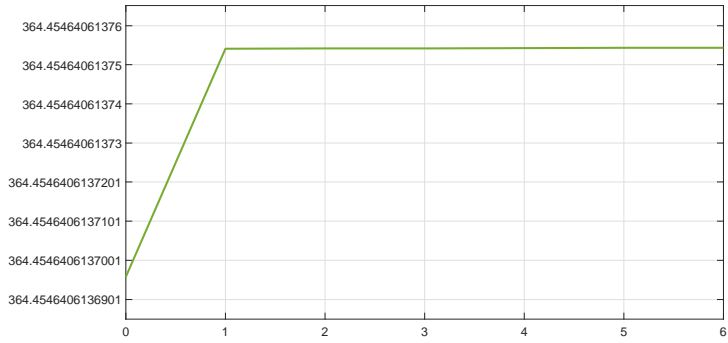}
	\caption[Evolution of the objective function ]{Evolution of the objective function from the initialisation $\mathcal{J}_{30}(u_0^0)=364.4 $ to the last iteration $\mathcal{J}_{30}(u_0^{50})=364.5$.}%
	\label{Ejemplo11}
\end{figure}

 \color{black}
 \newpage
\section{Conclusion, open problems and possible extensions}\color{black}
We make, in this conclusion, several concluding remarks and comments about possible generalisations and extensions of the results presented in this paper. For each of them, we try to present the arguments that have led us to the conclusion that other approaches were necessary in general.
\subsection{Regarding the regularity of the singular arc}
One of the main drawbacks of Theorem \ref{Prop1} is the regularity assumption on the singular arc. Namely, we obtain, for a maximiser $u_0$ of $\mathcal J_T$, the characterisation $f''(u_0)\leq 0$ only on the interior of the singular arc $\{0<u_0<1\}$. The method presented in this paper (two-scale expansions) strongly relies on the smoothness properties of the cut-off function $\theta$.

It may be tempting, in the one dimensional case, to overcome this difficulty arguing as in \cite{MNPE}: if we simply assume that the singular arc $\omega$ is measurable, but has positive measure, one can show that for every $K\in \N$ there exists $h_K\in L^2(\O)$, supported in $\omega$, that writes 
\begin{equation}h_K=\sum_{k\geq K}\alpha_{k,K}\cos(k\cdot)\,, \Vert h_K\Vert_{L^2}=1.\end{equation}Using the fact that $h_K$ only has high Fourier modes, one may hop for a two-scale expansion of the form 
$$h_K\approx \sum_{k\geq K}\alpha_{k,K}\left(h_{k}^0(k^2t,x,kx)+\frac1k h_k^1(k^2t,x,kx)\right).$$This is however \emph{a priori} prohibited by the problem of separation of phase: to obtain such a description, one needs welll-separated phases, in the sense of \cite{AllaireBriane}. In the context of Fourier series, this would require, at the very least, that $h_K$ should write as a lacunary Fourier series (typically, $h_K=\sum_{j=0}^\infty\alpha_{j,K}\cos(K^jx)$). However, for such lacunary Fourier series, Zygmund's theorem (see {\cite{Kovrizhkin2003}} for instance) prohibits that they have compact support, so that  admissible perturbations can not have this structure. This is a major drawback, and it is unclear whether or not one may be able to overcome this difficulty via a similar approach, or if an entirely new strategy needs to be devised.

Another approach would be to prove some regularity on $\omega$ ensuring that almost every of its point lie in its interior. This is satisfied for example if $\overline{u}_{0}$ is Riemann integrable (since, due to Lebesgue's characterisation of Riemann integrable functions, almost every point is a continuity point of $\overline{u}_{0}$). Riemann integrability is satisfied by BV functions. Unfortunately, we were not able to push the regularity further than $L^{\infty}$.

\subsection{Monostable non-linearities}
As seen in subsection \ref{e:f} of this paper, the numerical approach we propose, based on Theorem \ref{Prop1}, works for general monostable non-linearities. The theoretical tools are, however, not sufficient at this level to fully characterise optimisers. An interesting question would be to discuss whether or not optimisers in the monostable case are always bang-bang, are if some degeneracy zones can appear.
\color{black}

\subsection{The singular arc in higher dimensions}
  It may be plausible to \color{black} adapt the methods of Theorem \ref{Prop1} to obtain a characterisation of the singular arc analogous to that of Theorem \ref{Prop1} in the case $\O=\prod_{i=1}^N [0;a_i]$, $a_i>0$. To do so, the main difference with our proof would be to replace the initial perturbation $\theta(x)\cos(kx)$ with $\prod_{i=1}^N \theta(x)\cos(k x_i)$.

\subsection{Rearrangement inequalities for other types of boundary conditions}
In this work, we mostly dealt with the case of Neumann boundary conditions in the one-dimensional case. We ought to note two things: first, the proof of theorem \ref{Prop1} should hold in the case of Dirichlet or of Robin boundary conditions, provided the functions $\cos(k\cdot)$, in the proof, are replaced with the Dirichlet or Robin eigenfunctions of the laplacian in the interval. Second, regarding theorem \ref{Th:Convex}, the same type of results can be obtained in a straightforward manner for the case of Dirichlet boundary conditions, by applying directly \cite{Bandle}. The case of Robin boundary conditions may be encompassed by using the recent Talenti inequalities obtained in this case in \cite{Alvino2019ATC}. 
Addressing the problem on the full line $\R$ is more tricky, since in this case even the existence of a maximiser is unclear. We plan on investigating such matters in future works.

\bibliographystyle{abbrv}
\nocite{*}
\bibliography{BiblioMNT}

\end{document}